\documentclass{amsart}
\usepackage{calligra}
\usepackage{graphicx}
\usepackage{amsfonts}
\usepackage{amssymb,amsmath,nicefrac,amsthm}
\usepackage{color}
\usepackage{bbm}
\usepackage{float}
\usepackage{anysize}
\usepackage{multirow}
\usepackage{natbib}
\usepackage{mathtools}
\RequirePackage{hyperref}
\usepackage[top=1in,bottom=1in,left=1in,right=1in]{geometry}

\def\U{\mathcal{U}}
\def\UJ{\U_J}
\def\PP{\mathbb{P}}
\def\bb{\mathbf{b}}
\def\BB{\mathbf{B}}
\def\AA{\mathbf{A}}
\def\EE{\mathbb{E}}
\def\R{\mathbb{R}}

\def\a{\alpha}
\def\b{\beta}
\def\g{\gamma}

\def\e{\epsilon}

\def\l{\lambda}

\def\bul{\scriptscriptstyle \bullet}
\newcommand{\one}[1]{{\mathbbm{1}}\left\{{#1}\right\}}
\def\phat{\widehat{\phi}}
\def\phatj{\phat_J}

\def\pstarj{\phi^*_J}

\newcommand\T{\rule{0pt}{2.6ex}}
\newcommand\B{\rule[-1.2ex]{0pt}{0pt}}

\newtheorem{theorem}{Theorem}
\newtheorem{lemma}{Lemma}

\newtheorem{corollary}{Corollary}

\theoremstyle{definition}

\theoremstyle{remark}

\theoremstyle{theorem}
\makeatletter
\newtheorem*{rep@theorem}{\rep@title}
\newcommand{\newreptheorem}[2]{%
\newenvironment{rep#1}[1]{%
 \def\rep@title{#2 \ref{##1}}%
 \begin{rep@theorem}}%
 {\end{rep@theorem}}}
\makeatother

\newreptheorem{theorem}{Theorem}
\newreptheorem{lemma}{Lemma}
\newreptheorem{corollary}{Corollary}

\title[Model selection in sparse generalized linear models]{Bayesian
  model choice and information criteria \linebreak in sparse
  generalized linear models} 
\author{Rina Foygel\and Mathias Drton}
\address{Department of Statistics, The University of Chicago, Chicago,
  IL, U.S.A.}
\email{rina@uchicago.edu}
\email{drton@uchicago.edu}

\begin{document}

\begin{abstract}
  We consider Bayesian model selection in generalized linear models
  that are high-dimensional, with the number of covariates $p$ being
  large relative to the sample size $n$, but sparse in that the number
  of active covariates is small compared to $p$.  Treating the
  covariates as random and adopting an asymptotic scenario in which
  $p$ increases with $n$, we show that Bayesian model selection using
  certain priors on the set of models is asymptotically equivalent to
  selecting a model using an extended Bayesian information criterion.
  Moreover, we prove that the smallest true model is selected by
  either of these methods with probability tending to one.  Having
  addressed random covariates, we are also able to give a consistency
  result for pseudo-likelihood approaches to high-dimensional sparse
  graphical modeling.  Experiments on real data demonstrate good
  performance of the extended Bayesian information criterion for
  regression and for graphical models.
\end{abstract}

\keywords{Bayesian information criterion, Bayesian model selection,
  generalized linear model, graphical model, Ising model, logistic
  regression, sparse regression}

\maketitle

\section{Introduction}
\label{sec:introduction}

Information criteria provide a principled approach to a wide variety
of model selection problems.  The criteria strike a balance between
the fit of a parametric statistical model, measured by the maximized
likelihood function, and its complexity, measured by a penalty term
that involves the dimension of the model's parameter space.  The two
classical criteria are Akaike's information criterion (AIC)
\citep{Akaike:1974}, which targets good predictive performance, and the
Bayesian information criterion (BIC) introduced in
\citet{Schwarz:1978}, which is motivated by a connection to fully
Bayesian approaches to model determination and which has been proven
to enjoy consistency properties in a number of settings.  If we call a
model ``true'' if it contains the underlying data-generating
distribution, then consistency refers to selection of the smallest
true model in a suitable large-sample limit.  In this paper we will be
concerned with the BIC for generalized linear models with random
covariates in a sparse high-dimensional setting, where the number of
covariates is large but only a small fraction of the covariates is
related to the response.  Our main results show that extensions of the
BIC are consistent, and are accurate approximations of Bayesian
procedures, in asymptotic scenarios where the number of covariates
grows with the sample size.  The results can be interpreted either as
giving a precise Bayesian motivation for recently-proposed information
criteria, or as proving that Bayesian procedures enjoy favorable
frequentist properties.  In particular, our work gives uniform error
bounds for Laplace approximations to the large number of marginal
likelihood integrals arising in a Bayesian treatment of sparse
high-dimensional generalized linear models, and results in consistent
model selection for high-dimensional graphical models with binary
variables (the Ising model).

\subsection{Classical theory}
 
Suppose we observe a sample of $n$ observations for which we consider
the parametric model $\mathcal{M}$ with log-likelihood function
$\ell_{[n]}(\theta)$.  Then, written in the most commonly encountered form, the BIC for this model is
\[
\mathrm{BIC}(\mathcal{M})= -2\, \ell_{[n]}(\widehat{\theta}_{\mathcal{M}}) +
\dim(\mathcal{M})\cdot \log(n)\;, 
\]
with a lower value being desirable. Here $\dim(\mathcal{M})$ is the
dimension of the model's parameter space $\Theta_\mathcal{M}$, and
\[
\widehat{\theta}_{\mathcal{M}}=\arg\max_{\theta\in\Theta_\mathcal{M}}
\,\ell_{[n]}(\theta)  
\]
is the maximum likelihood estimator (MLE) in model $\mathcal{M}$.  The
classical large-sample theory underlying the BIC considers a finite
family of competing models that is closed under intersection, and for
which strict inclusion implies strictly lower dimension.  In order to
prove consistency of the BIC, it then suffices to make pairwise
comparisons between models $\mathcal{M}_1$ and $\mathcal{M}_2$ showing
that, with asymptotic probability one,
$\mathrm{BIC}(\mathcal{M}_1)<\mathrm{BIC}(\mathcal{M}_2)$ if either
(i) $\mathcal{M}_1$ is true and $\mathcal{M}_2$ is not, or (ii) both
$\mathcal{M}_1$ and $\mathcal{M}_2$ are true but
$\dim(\mathcal{M}_1)<\dim(\mathcal{M}_2)$.  In case (i), a proof shows
that the difference in the likelihood terms in the BIC outgrows the
logarithmic penalty term as $n\to\infty$, whereas in case (ii) the
logarithmic term outgrows the difference in log-likelihood values,
which remains bounded in probability; compare, for instance,
\citet{Nishii:1984}, \cite{Haughton:1988} or monographs on model selection and
information 
criteria such as \citet{Burnham:2002,Claeskens:2008,Konishi:2008}.

The penalty term appearing in the BIC is only one of many
possible choices to balance model fit and complexity in a way that
leads to consistency.  However, the logarithmic dependence on the
sample size makes a connection to Bayesian approaches.  Consider the
prior $f_\mathcal{M}(\theta)$ for the parameter $\theta$ in model
$\mathcal{M}$, and write $P(\mathcal{M})$ for the prior probability of
$\mathcal{M}$.  Then the posterior probability of $\mathcal{M}$ is
proportional to 
\begin{equation}\label{eq:BayesML_def}
  P(\mathcal{M})\cdot \int_{\theta}  \exp\{\ell_{[n]}(\theta)\}\,
  f_\mathcal{M}(\theta)\;d\theta,
\end{equation}
where the integral is commonly referred to as the marginal likelihood.
In well-behaved models, for large $n$, the integrand in the marginal
likelihood takes large values only in a neighborhood of the MLE
$\widehat{\theta}_{\mathcal{M}}$.  Moreover, in such a neighborhood the
log-likelihood $\ell_{[n]}(\theta)$ can be approximated by a quadratic function,
 while the prior $f_\mathcal{M}(\theta)$
is approximately constant.  Evaluating the
resulting Gaussian integral reveals that the logarithm of the marginal
likelihood equals $-\nicefrac{1}{2}\cdot\mathrm{BIC}(\mathcal{M})$
plus a remainder term that is bounded in probability when the $n$
observations are drawn from a distribution in $\mathcal{M}$ and the
sample size $n$ tends to $\infty$.  The remainder term can be
estimated to be
\[
\frac{1}{2}\dim(\mathcal{M})\log(2\pi)+\log
f_\mathcal{M}(\widehat{\theta}_{\mathcal{M}})-
\frac{1}{2}\log\det\left(\frac{1}{n}H_{[n]}(\widehat\theta_\mathcal{M})\right)
+\mathbf{O}_P(n^{-1/2}),
\]
where $H_{[n]}$ is the Hessian of the negated log-likelihood function (and
scales with $n$).  The work of \citet{Haughton:1988} provides a
rigorous probabilistic treatment of this Laplace approximation to the
marginal likelihood in the general setting of smooth (or curved)
exponential families.  Considering a finite family of models, it
suffices to treat one model at a time in this analysis.

\subsection{Recent extensions of the BIC}
\label{subsec:intro-ebic}

In the last decade, new applications of the BIC have emerged in
problems of selecting sparse models for high-dimensional data,
including problems such as tuning parameter selection in Lasso and
related regularization procedures; see e.g.~\citet{Zou:2007}.
Following a proposal from \citet{Bogdan:2004}, the work of
\citet{Chen:2008} treats an extended BIC for variable selection in
sparse high-dimensional linear regression with deterministic
covariates. The extension allows for a more stringent penalty to
address the (ordinary) BIC's tendency to select overly large models in
this setting.  The asymptotic scenario underlying the theoretical
analysis of the new criterion allows for subexponential growth in the
number of covariates $p$ as a function of the sample size $n$, with a
bound on the number of covariates that appear in the true mean
function.  The main result of \citet{Chen:2008} shows variable
selection consistency under these asymptotics.  \citet{Chen:2011}
extend the results to generalized linear models (GLMs), and further
improvements for linear regression are given in \citet{Luo:2011} and
in \cite{Zhang:2010}.  Consistency in Gaussian graphical models has
been studied by \citet{Foygel:2010} and \citet{Gao:2011}.  Composite
likelihood-based criteria are treated by \citet{Gao:2010}.  The main
difficulty in showing consistency in these high-dimensional settings
is the need to control a diverging number of models.  We remark that a
related study of the ordinary BIC that focuses on pairwise model
comparisons is given by \citet{Moreno:2010}.

In the literature on the extended BIC, the key idea for treating the
high-dimensional setting is to augment the BIC with an informative
prior on models.  In the regression setting, which is our focus in
this paper, the competing models correspond to different subsets of
covariates.  If $p$ denotes the number of covariates, then a model
corresponds to a subset ${J}\subset [p]\coloneqq\{1,\dots,p\}$, and the
extended BIC is based on the prior
\begin{equation}
  \label{eq:p_choose_|J|}
  P(J)=\frac{1}{p+1}{p\choose |J|}^{-1}
\end{equation}
that gives equal probability to each model size, and to each model given
the size.  Clearly, this priors favors an individual small model
over an individual model of moderate size.  More generally, priors
of the form
\begin{equation}
  \label{eq:chen_prior}
  P(J)\propto {p\choose |J|}^{-\gamma}\cdot \one{|J|\leq q}
\end{equation}
with a hyperparameter $\gamma\in[0,1]$ have been considered to allow
one to interpolate between the classical BIC of \citet{Schwarz:1978},
obtained for $\gamma=0$, and the prior in (\ref{eq:p_choose_|J|})
given by $\gamma=1$, while at the same time invoking an upper bound
$q$ on the number of covariates.  Some of our later asymptotic results
suggest that it can be useful to consider $\gamma>1$ if the number of
covariates $p$ is very large.  Priors of this form have also been
shown to be useful in fully Bayesian approaches to high-dimensional
regression; compare \citet{Scott:2010} who motivate this and related
priors in a construction that includes each covariate with a fixed
probability that itself is given a Beta prior.

Writing $J$ for a particular subset of the available covariates,
inclusion of the prior from (\ref{eq:chen_prior}) into the information
criterion yields the \emph{extended BIC} (EBIC)
\[
\mathrm{BIC}_{\g}(J)= -2 \ell_{[n]}(\widehat{\theta}_J) +
|J|\cdot \log(n) +2\g |J|\cdot\log (p)\,. 
\]
Under suitable conditions on the design matrix (ensuring, for
instance, that removing a covariate from the smallest true model will
substantially lower the achievable likelihood), 
\citet{Chen:2008,Chen:2011} show consistency of the EBIC in
both the normal linear regression and the univariate GLM setting with
fixed (i.e., non-random) covariates.  Consistency holds as long as
$\g>1-\frac{1}{2\kappa}$, where $\kappa$ determines the rate of growth
of the number of covariates, and thus the space of possible models,
with $p=\mathbf{O}(n^{\kappa})$.

While our focus is entirely on consistency properties of model
selection procedures for high-dimensional regression, we should
mention that other properties are of interest and have been studied.
For instance, \citet{Shao:1997} considers a similar problem, proving
results about selection of the best predictive model, rather than
consistency in variable selection.  \citet{Jiang:2007} points out that
in applied settings, the question of variable selection is not always
well-defined due to many coefficients that are approximately rather
than exactly zero, and considers the problem of estimating the true
parameter vector under the assumption of approximate sparsity.  This
paper uses a prior on models similar to the one
in~(\ref{eq:p_choose_|J|}).

\subsection{Recent work on Bayesian model selection in
  high-dimensional settings}

Several recent papers have examined the properties of Bayesian model
selection in scenarios where the model size may be large relative to
the sample size.  For instance, a consistency result for Bayesian
linear regression has recently been obtained by \citet{Shang:2011}.
This work focuses on a specific Bayesian model that assumes the
regression coefficients to follow a particular prior distribution that
is a mixture of a normal distribution and a point mass at zero. There
has also been work on the problem of choosing between a pair of
models, including a recent article by \citet{Kundu:2011} that shows
consistency of Bayesian pairwise model comparison under a flexible,
non-parametric noise model. These results are not directly comparable
to the problem discussed here, where we consider the problem of
searching for the smallest true model from among a combinatorially
large set of possible sparse models in a generic regression setting.
 
\subsection{New results}

With a penalty term reflecting a particular type of prior on models,
the EBIC has a clear Bayesian motivation.  However, it is not
immediately clear that the EBIC and fully Bayesian approaches using
the same prior on models should lead to asymptotically equivalent
model choice in a high-dimensional asymptotic scenario that has the
number of covariates $p$ grow with the sample size $n$.  Our first
main result, Theorem~\ref{thm:Bayes=BIC}, addresses this issue for
generalized linear models and shows that such equivalence indeed
occurs at a fairly general level.  More precisely, our result shows
that a Laplace approximation to the marginal likelihood of each one of
a growing number of models results into errors that are, with high
probability, uniformly bounded as
$\mathbf{O}\big(\sqrt{{\log(np)}/{n}}\big)$.

Our second main result, Theorem~\ref{thm:BIC_consistent}, provides a
consistency result for the EBIC.  The result is very closely related
to those of \citet{Chen:2011}.  The primary difference is that we
consider random rather than deterministic covariates and allow for
unbounded covariates, subject to a moment condition, in some special
cases such as logistic regression.  Consistency is proven under the
same conditions that we use to obtain the equivalence of the EBIC and
fully Bayesian model selection.  Combining the two Theorems yields
Corollary~\ref{cor:BayesML_consistency}, which states consistency for
fully Bayesian model determination.

Theorem~\ref{thm:BIC_consistent} also allows us to obtain consistency
results for pseudo-likelihood methods in graphical model selection,
where regressions are performed to model each node's dependence on the
other nodes in the graph (``neighborhood selection'' for each node),
and these neighborhoods are then combined to hypothesize a sparse
graphical model; see \citet{Meinshausen:2006} and
\citet{Ravikumar:2010}. Since in each regression, the covariates
consist of random observations from the potential ``neighbors'' of the
node in question, it is crucial that our analysis of consistency
allows for random covariates. Furthermore, for consistent model
selection, the neighborhood selection procedure must succeed
simultaneously for each node, and therefore we make use of the
explicitly-calculated finite-sample bounds in
Theorem~\ref{thm:BIC_consistent}.  Our results for the graphical Ising
model are given in Theorem~\ref{thm:Ising}.

The remainder of this paper is organized as follows. We begin by
defining the setting and introducing notation, in Section
\ref{sec:prelim}.  In Section \ref{sec:BayesML}, we discuss Bayesian
model selection and the Laplace approximation to the marginal
likelihood. Our result on the consistency of the EBIC for regression
is given in Section \ref{sec:EBIC}, where we also present experiments
on real data that show good performance of the EBIC in practice. In
Section \ref{sec:graphs}, we turn to graphical models and present
theoretical and empirical results showing the consistency of the EBIC
for reconstructing a sparse graph based on a neighborhood selection
procedure. We outline the proofs for Theorems~\ref{thm:Bayes=BIC}
, \ref{thm:BIC_consistent}, and~\ref{thm:Ising} in Section \ref{sec:proofs}, and give the
full proofs in the Appendix. Finally, in Section \ref{sec:conclusion},
we discuss our results and outline directions for future work.

\section{Setup and assumptions}
\label{sec:prelim}

We treat generalized linear models in which the observations of the
response variable follow a distribution from a univariate exponential
family with densities
\[
p_\theta(y)\,\propto\, \exp\left\{y \cdot
  \theta-\bb(\theta)\right\}\,, \qquad \theta\in\Theta=\R,
\]
where the density is defined with respect to some measure on $\R$. 
More precisely, the observations of the response are independent
random variables $Y_1,\dots,Y_n$, with $Y_i\sim p_{\theta_i}$.  The
vector of natural parameters
$\boldsymbol{\theta}=(\theta_1,\dots,\theta_n)^T$ is assumed to lie in
the linear space spanned by the columns of a design matrix
$X=(X_{ij})\in\R^{n\times p}$, that is, $\boldsymbol{\theta} = X\phi$
for some parameter vector $\phi\in\R^p$.  Our focus is on the case of
random covariates.  Let $X_{\bul j}$  be the $n$-dimensional vector  
of observed values for the $j$th covariate (the $j$th column of the
design matrix $X$), 
 and write $X_{i\bul}$ for the $p$-dimensional
covariate vector in the $i$th row of the design matrix $X$.  Then we
assume $X_{1\bul},\dots,X_{n\bul}$ to be independent and
identically distributed random vectors.

Our results treat a sparsity scenario in which the joint distribution
of $Y_1,\dots,Y_n$ is determined by a true parameter vector
$\phi^*\in\R^p$ supported on a (small) set $J^*\subset[p]$, that is,
$\phi^*_j\not=0$ if and only if $j\in J^*$.  Our interest is in
recovery of the set $J^*$.  To this end, we consider the different
submodels given by the linear spaces spanned by subsets $J\subset[p]$
of the columns of the design matrix $X$.  We will use $J$ to denote
either an index set for covariates or the resulting model, for
convenience.  Finally, we denote subsets of covariates by
$X_{iJ}=(X_{ij})_{j\in J}$, where $J\subset [p]\coloneqq\{1,\dots,p\}$.

\subsection{Assumptions}

We will be concerned with asymptotic questions in a scenario in which
$n\rightarrow\infty$ and the number of covariates $p=p_n$ is allowed
to grow.  Let $\kappa_n=\log_n(p_n)$, and let
$\kappa=\lim\sup\kappa_n\in[0,\infty]$.  Subsequently, we will
suppress the sample size index and write $p$ rather than $p_n$.  Our
theorems apply to either one of the following cases
 (recall that $X_{1j},\dots,X_{nj}$ are identically distributed):
\begin{itemize}
\item[(A1)] The covariates are bounded (or bounded with probability
  one), that is, there is a constant 
  $\AA<\infty$ such that, $|X_{1j}|\leq \AA$ for $j=1,\dots,p$.
\item[(A2)] There is an even integer $K>2\kappa$ (in particular,
  $\kappa<\infty$), for which the covariates have moments bounded as
  $\EE\left[|X_{1j}|^{6K}\right]\leq \AA_K<\infty$ for $j=1,\dots,p$.
  Moreover, for all $t>0$, it holds that
  \[
  \BB_K(t)\coloneqq \sup_j \EE_{X_{1j}}\left[\sup_{|\theta|\leq
      t|X_{1j}|}\left|\bb'''(\theta)\right|^{2K}\right]<\infty\;.
  \]
\end{itemize}
We remark that the condition on the third derivative of the cumulant
generating function $\bb$ holds for logistic and Poisson models, as
well as the normal model with known variance.  In addition, we assume
the following five conditions:
\begin{itemize}
\item[(B1)] The growth of $p$ is subexponential, that is,
  $\log(p)=\mathbf{o}(n)$.
\item[(B2)] The size of the true model given by the cardinality of the
  support $J^*$ of the true parameter vector $\phi^*$ is bounded as
  $|J^*|\leq q$ for a fixed integer $q\in\mathbb{N}$.
\item[(B3)] All small sets of covariates have second moment matrices
  with bounded eigenvalues, that is, for some fixed finite constants
  $a_1,a_2>0$, it holds that $a_1\mathbf{I}_J\preceq
  \EE\left[X_{1J}^{}X_{1J}^T\right]\preceq a_2\mathbf{I}_J$ for all $|J|\leq
  2q$.
\item[(B4)] The norm of the true signal is bounded, namely,
  $\|\phi^*\|_2\leq a_3$ for a fixed constant $0<a_3<\infty$.
\item[(B5)] The small true coefficients have bounded decay such that
  \[
  \sqrt{\frac{\log(np)}{n}}=\mathbf{o}\left(\min\left\{\left|\phi^*_j\right|
      :j\in J^*\right\}\right).
  \]
\end{itemize}

\subsection{Comparison to assumptions used in existing work}

We compare the above assumptions to those used by \citet{Chen:2011},
who show that the EBIC is consistent for univariate GLMs with
fixed covariates. In this work by Chen and Chen, only bounded covariates are
considered --- that is, our (A1) scenario. Our conditions (B1), (B2),
and (B4) appear (explicitly or implicitly) in their work as well.
Condition (B5) appears in a stronger form in their work, where they
assume that $\phi^*$ is fixed and therefore its minimal nonzero value
is bounded from below by some constant.

A crucial difference lies in our assumption (B3). The analogous
condition of \citet{Chen:2011} requires that, for some positive finite
$\l_1$ and $\l_2$, $\l_1\mathbf{I}_{J}\preceq n^{-1}
H_J(\phi_J)\preceq \l_2 \mathbf{I}_J$, for any $J\supset J^*$ with
$|J|\leq 2q$ and any $\phi_J$ in a neighborhood of $\pstarj$. Here
$H_J(\cdot)\coloneqq \left(H_{[n]}(\cdot)\right)_{J,J}$ 
is the Hessian of the negative log-likelihood function (restricted to rows and columns
corresponding to the covariates in $J$). Note that $H_J(\cdot)$
 depends on the design matrix $X_J$ for the given set of
covariates $J$. Since we work in the setting of random covariates, we
cannot make this assumption on the empirical design matrix, and
therefore use condition (B3), which is a weaker assumption on the
distribution of the covariates.

\section{Bayesian model selection}
\label{sec:BayesML}

Observing the independent random vectors
$(X_{1\bul},Y_1),\dots,(X_{n\bul},Y_n)$, the likelihood function of
the considered GLM is
\[
L_{[n]}(\phi)=\exp\left\{\ell_{[n]}(\phi)\right\}=\exp\left\{\sum_{i=1}^n \ell_i(\phi)\right\}=\exp\left\{\sum_{i=1}^n
  Y_i\cdot X_{i\bul}^T\phi-\bb(X_{i\bul}^T\phi)\right\}\;,
\]
with $\ell_i(\phi)$ being the log-likelihood function based on the
$i$th observation $(X_{i\bul},Y_i)$.  Let $P(J)$ be the prior
probability of model $J\subset[p]$, and let $f_J(\phi_J)$ be a prior
density on the model's parameter space $\R^J$.  The unnormalized
posterior probability of model $J$ is then
\[
\mathrm{Bayes}(J)=
P(J)\cdot\int_{\phi_J\in\R^J} L_{[n]}(\phi_J)f_J(\phi_J)\; d\phi_J\;,
\]
where, with some abuse of notation, we write $L_{[n]}(\phi_J)$ for the
likelihood function of the model given by $J\subset[p]$, that is,
\[
L_{[n]}(\phi_J)=\exp\left\{\sum_{i=1}^n
  \ell_i(\phi_J)\right\}=\exp\left\{\sum_{i=1}^n Y_i\cdot
  X_{iJ}^T\phi_J-\bb(X_{iJ}^T\phi_J)\right\}.
\]

Our interest is now in the frequentist properties of the Bayesian
model selection procedure that chooses a model $J$ by maximizing the
(unnormalized) posterior probability $\mathrm{Bayes}(J)$.  Assuming
that observations are drawn from a distribution in the GLM, we ask
the following two questions.  First, is the Bayesian procedure
consistent, that is, will it choose the smallest true model in the
large-sample limit?  Second, how can we approximate the marginal
likelihood integral appearing in $\mathrm{Bayes}(J)$, without
introducing approximation errors that might change which model is
selected?  In the classical scenario with a fixed number of
covariates $p$, when considering a growing sample size $n$, the
answers to the above questions are tied together.  Under suitable
conditions, consistency of the Bayesian procedure can be established
by proving that, for large samples, it selects the same model as the
consistent BIC or the more accurate approximation obtained by applying
the Laplace approximation to the marginal likelihood integral.  We
will show these same connections to exist in sparse high-dimensional
settings.

Theorem~\ref{thm:Bayes=BIC} below states that, under appropriate
conditions, the Laplace approximation to marginal likelihood integrals
remains uniformly accurate across large spaces of models.  This result
is obtained under an upper bound $q$ on the model size.  As discussed
in Section~\ref{subsec:intro-ebic} in the introduction, we give
special emphasis to a particular class of prior distributions on the
set of models, namely, priors of the form
\begin{equation}
  \label{eq:prior}
  P(J)\propto {p\choose |J|}^{-\g}\cdot \one{|J|\leq q}\;, \qquad J\subset[p],
\end{equation}
for some $\g\geq 0$.  
We write
$\mathrm{Bayes}_\gamma(J)$ for the unnormalized posterior probability
associated with the choice of prior $P(J)={p\choose |J|}^{-\g}\cdot
\one{|J|\leq q}$, where we suppress the normalizing constant in the prior for convenience.
  Then we show that, for sufficiently large $n$, the
event
\[
\arg\min_{|J|\leq q} \mathrm{BIC}_{\g}(J) = \arg\max_{|J|\leq q}
\mathrm{Bayes}_{\g}(J)\;.
\]
occurs with high probability.  In other words, the EBIC
\begin{equation}
\label{eq:EBIC-1}
\mathrm{BIC}_{\g}(J)=-2\log L_{[n]}(\phatj)+|J|\log(n)+2\g|J|\log(p)\;.
\end{equation}
yields an approximation to the Bayesian posterior probability that is
accurate enough for the resulting model selection procedures to be
asymptotically equivalent.  In fact, Theorem~\ref{thm:Bayes=BIC}
states a stronger result according to which $\mathrm{Bayes}_\g(J)$ is
approximated up to a constant by $\mathrm{BIC}_\g(J)$. 
Finally, we prove consistency of
the EBIC in Section~\ref{sec:EBIC}.  In combination with the results
of this section, we obtain a proof of the consistency of the Bayesian
model selection procedure.

We now give the precise statement of the points just outlined.  We
adopt the notation $a=b(1\pm c)$ to conveniently express that $a$
belongs to the interval $[b(1-c),b(1+c)]$.

\begin{theorem}\label{thm:Bayes=BIC}
  Assume that conditions~(B1)-(B5) hold, and that either assumption
  (A1) or (A2) holds.  Moreover, assume the following mild conditions
  on the family of priors $\left(f_J:J\subset[p],|J|\leq q\right)$,
  which require the existence of constants $0<F_1,F_2,F_3<\infty$
  such that, uniformly for all $|J|\leq q$, we have
  \begin{itemize}
  \item[(i)] an upper bound on the priors: 
    \[
    \sup{}_{\phi_J} f_J(\phi_J)\leq F_1<\infty,
    \]
  \item[(ii)] a lower bound on the priors over a compact set:
    \[
    \inf{}_{\|\phi_J\|_2\leq R+1} f_J(\phi_J)\geq F_2>0,
    \]
    where $R$ is a function of the constants in assumptions (A1) or (A2) and (B1)-(B5), defined in the proofs,
  \item[(iii)] a Lipschitz property on the same compact set:
    \[
    \sup{}_{\|\phi_J\|_2\leq R+1} \left\|\nabla
      f_J(\phi_J)\right\|_2\leq F_3<\infty.
    \]
  \end{itemize}
  Then there is a constant $C$, no larger than
  $4F_3F_2^{-1}\l_1^{-\nicefrac{1}{2}}+2q\l_3\l_1^{-\nicefrac{3}{2}}+2$, 
  such that, for sufficiently large $n$, the event that 
  \begin{equation*}
    \mathrm{Bayes}(J) 
    =P(J)\cdot L_{[n]}(\phatj)f_J(\phatj)\cdot
    \left|H_J(\phatj)\right|^{-\nicefrac{1}{2}}(2\pi)^{\nicefrac{|J|}{2}}\cdot
    \left(1\pm C\sqrt{\frac{\log(np)}{n}}\right)  \text{ for all }  |J|\leq q
  \end{equation*}
  occurs with
  probability at least $1-(np)^{-1}$ under (A1), and with probability at
  least $1-(np)^{-1}-4K^{K+1}n^{-\frac{K-2\kappa}{2}}$ under (A2).
  In particular, for the (unnormalized) prior $P(J)= {p\choose
    |J|}^{-\gamma}\cdot\one{|J|\leq q}$, it holds that 
  \begin{equation*}
    \left|\log\left(\mathrm{Bayes}_{\g}(J)\right)-
      \left(-\tfrac{1}{2}\mathrm{BIC}_{\g}(J)\right)\right|\leq
    C_1\;,
  \end{equation*}
  where $C_1$ is a constant no larger than
  $\tfrac{q}{2}\log(2\pi)+\g
    q\log(2q)+q\log\max\{\l_1^{-1},\l_2\}+\log\max\{F_1,F_2^{-1}\}+1$.
\end{theorem}

The proof of this theorem is given in Section~\ref{sec:proofs}. 
The constant $R$ appearing in conditions (ii) and (iii) on the family of priors arises in the proof, where we show that with high probability, the MLEs $\widehat{\phi}_J$ for all sparse models $J$ will lie inside a ball of radius $R$ centered at zero.

\section{Consistency of the extended Bayesian information
  criterion}\label{sec:EBIC} 

Let $J^*$ be the smallest true model; recall Section~\ref{sec:prelim}.
We now show that the extended BIC from (\ref{eq:EBIC-1}) satisfies
that, with high probability,
$\mathrm{BIC}_{\g}(J)>\mathrm{BIC}_{\g}(J^*)$ for all $J\neq J^*$ with
$|J|\leq q$, as long as the penalty on model complexity is
sufficiently large.  Specifically, we require $\g>1-\frac{1}{2\kappa}$
(and $\g\geq 0$), where $\kappa=\lim\sup
\kappa_n=\lim\sup\log_n(p_n)\in[0,\infty]$.

Our main consistency result, stated next, is very similar to the
consistency results of \citet{Chen:2011}
but treats random instead of deterministic covariates.

\begin{theorem}\label{thm:BIC_consistent} 
  Assume that conditions~(B1)-(B5) hold, and that either assumption
  (A1) or (A2) holds.  Choose three scalars $\a,\b,\g$ to satisfy
 \begin{equation*}
   \left\{\text{\begin{tabular}{ll}\T\B
         $\g>1-\tfrac{1}{2\kappa}+\b+\tfrac{\a}{\kappa}$,&if 
         $\kappa>0$,\\ 
         \T\B$\a\in\left(0,\tfrac{1}{2}\right)$ and $\b>0$,&if
         $\kappa=0$.\\\end{tabular}}\right.
 \end{equation*}
 Then, for sufficiently large $n$, the event 
 \begin{equation}\label{eq:BIC}
   \mathrm{BIC}_{\g}(J^*) \leq 
   \left(\min_{J\neq J^*,|J|\leq q}\mathrm{BIC}_{\g}(J)\right)-
   \log(p)\cdot\left(\g-\left(1-\frac{1}{2\kappa}+\b+
       \frac{\a}{\kappa}\right)\right)\;
 \end{equation}
 occurs with probability at least $1-n^{-\a}p^{-\b}$ under (A1) or at
 least $1-4K^{K+1}n^{-\frac{K-2\kappa}{2}}-n^{-\a}p^{-\b}$ under (A2).
 In particular, the EBIC is  consistent for model
 selection, whenever $\g>1-\frac{1}{2\kappa}$.
\end{theorem}

Combining Theorem~\ref{thm:BIC_consistent} with
Theorem~\ref{thm:Bayes=BIC}, which showed the equivalence of
EBIC-based and Bayesian model selection, we obtain the following
corollary.

\begin{corollary}\label{cor:BayesML_consistency}
  Assume that conditions~(B1)-(B5) hold, and that either assumption
  (A1) or (A2) holds. Choose three scalars $\a,\b,\g$ to satisfy
  \begin{equation*}
    \left\{\text{\begin{tabular}{ll}\T\B$\g>1-
          \tfrac{1}{2\kappa}+\b+\tfrac{\a}{\kappa}$,&if $\kappa>0$,\\
          \T\B$\a\in\left(0,\tfrac{1}{2}\right)$ and $\b>0$,&if $\kappa=0$.\\
        \end{tabular}}\right.
  \end{equation*}
  Then, for sufficiently large $n$, with probability at least
  $1-n^{-\a}p^{-\b}$ under (A1) or at least
  $1-4K^{K+1}n^{-\frac{K-2\kappa}{2}}-n^{-\a}p^{-\b}$ under (A2),
  \begin{equation*}
    \mathrm{Bayes}_{\g}(J^*) > \min_{J\neq J^*,|J|\leq
      q}\mathrm{Bayes}_{\g}(J)\;. 
  \end{equation*}
  In particular, Bayesian model selection is consistent, whenever
  $\g>1-\frac{1}{2\kappa}$.

\end{corollary}

\subsection{Selecting from a set of candidate models}

In practice, it is not computationally feasible to calculate either
the Bayesian marginal likelihood or the EBIC for every possible sparse
model, since even if the model size bound $q$ is relatively small, the
number of possible models is very large, on the order of $p^q$.
Furthermore, the size $q$ of the smallest true model is not known in
general. Typically, the BIC (or another selection criterion) is applied
only to a manageable number of candidate models, obtained via some
other method. In the sparse regression setting, the Lasso
\citep{Tibshirani:1996} has been demonstrated to be very effective at
recovering sparse linear and generalized linear models
\citep{Friedman:2010}.  The Lasso selects and fits a model by solving
the convex optimization problem
\begin{equation}
  \widehat{\phi}^{\rho}=\arg\min_{\phi\in\R^p} \Big\{- \sum_i \ell_i(\phi) +
  \rho\|\phi\|_1\Big\}\;,\label{eq:lasso}
\end{equation}
where $\|\phi\|_1= \sum_j |\phi_j|$ is the vector $1$-norm and
$\rho\geq 0$ is a penalty parameter. For an appropriate choice of
$\rho$ and under some conditions on the covariates and the signal, the
Lasso is known to be consistent for linear regression; compare Chapter
6 of \citet{Buhlmann:2011}.  The optimal choice of $\rho$ suggested by
theory depends on unknown properties of the distribution of the data,
and is therefore unknown in an applied setting. A common approach to
the problem is to fit the entire ``Lasso path'' of coefficient vectors
$\widehat{\phi}^{\rho}$ for $\rho$ in the range $[0,\infty)$, thus
producing a list of candidate sparse models $\{J_1,J_2,\dots\}$, and
then to select a model from this list using a technique such as
cross-validation or the BIC. By Theorem~\ref{thm:BIC_consistent}, with
probability near one (for large $n$),
$\mathrm{BIC}_{\g}(J^*)<\mathrm{BIC}_{\g}(J_m)$ for any sparse model
$J_m\neq J^*$ in the candidate set. Therefore, if the smallest true
model $J^*$ is in the candidate set, we will be able to find it with
high probability by applying the EBIC to every candidate model.

\subsection{Experiment for sparse
  logistic regression} 

We compare the BIC, the extended BIC with $\g=0.25$ and $\g=0.5$, and
10-fold cross-validation on the task of selecting a logistic model for
distinguishing between spam and legitimate emails. We compare also
to stability selection \citep{Meinshausen:2010}, a recent alternative approach
to the problem of sparse model selection that applies the Lasso repeatedly
to subsamples of the data, and then chooses covariates to include in the 
model based on whether they are ``stable'', that is, whether they appear
consistently over the repeated samples.

\subsubsection{Data and methods for model selection}
\label{sec:EmailMethods}

We used the \textsc{Spambase} data set from the UCI Machine Learning
Data Repository \citep{UCI_data}.\footnote{Available at
  \texttt{http://archive.ics.uci.edu/ml/datasets/Spambase}} The data
is drawn from 4,601 emails, and consists of a binary response (spam or
non-spam classification), along with predictors measuring the
frequency of certain words and characters in the email, and several
other predictive features, for a total of 57 real-valued covariates.
To create a challenging setting where the number of covariates is
large relative to the sample size, we first randomly sampled a subset
$S\subset\{1,\dots,4601\}$, for various sample sizes $n=|S|$. We then
created fake covariates by permuting the true features, in order to
allow $p$ to grow with $n$. We ran 100 repetitions of each of the
settings shown in Table~\ref{table:SpamDataSettings}.

\begin{table}[t]
\caption{Settings for the spam email experiment.}
\begin{center}\begin{tabular}{c||c|c|c}
$n$ &$p$ &\begin{tabular}{c}\# true\\
  features\end{tabular}&\begin{tabular}{c}\# permuted\\
  features\end{tabular}\\\hline 
$100$&$57\cdot 4$&$57$&$57\cdot 3$\\
$200$&$57\cdot 8$&$57$&$57\cdot 7$\\
$300$&$57\cdot 12$&$57$&$57\cdot 11$\\
$400$&$57\cdot 16$&$57$&$57\cdot 15$\\
$500$&$57\cdot 20$&$57$&$57\cdot 19$\\
$600$&$57\cdot 24$&$57$&$57\cdot 23$\\
\end{tabular}\end{center}
\label{table:SpamDataSettings}
\end{table}

Let $Y_i$ be the class label with $Y_i=1$ if email $i$ is spam, and
$Y_i=0$ otherwise.  Let $X_{ij}$ be the value of the $j$th covariate
for the $i$th email.
For each $(n,p)$ pair, we performed the following steps. We first drew
a subsample $S=\{i_1,\dots,i_n\}\subset\{1,\dots,4601\}$ uniformly at
random, and define the response vector to be
$(Y_{i_1},\dots,Y_{i_n})^T$. We then randomly chose permutations
$\sigma_1,\dots,\sigma_K$ of $\{1,\dots,n\}$, where $K$ is chosen to
obtain the desired total number of covariates, i.e.\ $p=57\cdot
(1+K)$. We define the design matrix
$$
\left(\begin{array}{ccc|ccc|c|ccc}
X_{i_1,1}&\dots&X_{i_1,57}&X_{i_{\sigma_1(1)},1}&\dots&X_{i_{\sigma_1(1)},57}&\dots&X_{i_{\sigma_K(1)},1}&\dots&X_{i_{\sigma_K(1)},57}\\
\dots&\dots&\dots&\dots&\dots&\dots&\dots&\dots&\dots&\dots\\
X_{i_n,1}&\dots&X_{i_n,57}&X_{i_{\sigma_1(n)},1}&\dots&X_{i_{\sigma_1(n)},57}&\dots&X_{i_{\sigma_K(n)},1}&\dots&X_{i_{\sigma_K(n)},57}\\
\end{array}\right)\;,$$
which contains one block of 57 true features, and $K$ blocks of 57 permuted (fake) features.

To evaluate the BIC, the EBIC, and cross-validation on this data, we
first generated models by applying the logistic Lasso with a range of
100 penalty-parameter values to the data, using the $\texttt{glmnet}$
package \citep{Friedman:2010} in $\textsc{R}$ \citep{R}. This produced
a list of 100 (possibly not distinct) support sets, 
$J_1,\dots,J_{100}$.
For the BIC and the EBIC, we refitted each candidate model 
$J_m$ using the function $\texttt{glm}$ in $\textsc{R}$, and applied $\mathrm{BIC}_\g$ with $\g=0.0,0.25,0.5$ to each candidate model, to select a single model for each $\mathrm{BIC}_\g$. We also applied 10-fold cross-validation, selecting the single model from the list of candidate models that minimizes average error on the test sets over the 10 folds.

Finally, for stability selection, we used the $\texttt{stabsel}$
function in the $\texttt{mboost}$ package \citep{R_mboost} in
$\textsc{R}$, with expected support set size $q=50$. As noted by
\citet{Meinshausen:2010}, changing the settings within a reasonable
range did not have a large effect on the output.

\subsubsection{Results} 

We evaluate the methods based on their ability to distinguish between
the 57 true and the remaining false (permuted) features.
Table~\ref{table:SpamData} and Figure~\ref{fig:SpamResults} show the
positive selection rate (PSR) and the false discovery rate (FDR) for
each of the five methods in this task, over the range of sample sizes.
As customary, PSR is defined as the proportion of true features
selected by the method, and FDR is the proportion of false positives
among all features selected by the method.

\begin{table}[t]
        \caption{Positive selection rate and false discovery rate in the spam email experiment.}
{\small
\begin{center}
\begin{tabular}{|@{\,}c@{\,}||r|r||r|r||r|r||r|r||r|r||r|r|}
\hline
&\multicolumn{2}{c||}{$n=100$}&\multicolumn{2}{c||}{$n=200$}&\multicolumn{2}{c||}{$n=300$}&\multicolumn{2}{c||}{$n=400$}&\multicolumn{2}{c||}{$n=500$}&\multicolumn{2}{c|}{$n=600$}\\
&PSR&FDR&PSR&FDR&PSR&FDR&PSR&FDR&PSR&FDR&PSR&FDR\\\hline\hline
BIC$_{0.0}$ & 14.12 & 10.95 & 19.37 & 18.04 & 23.39 & 20.04 & 26.40 &
20.79 & 30.79 & 22.86 & 33.46 & 19.81 \\ \hline BIC$_{0.25}$ & 8.65 &
2.18 & 11.33 & 0.92 & 15.11 & 1.49 & 17.42 & 1.00 & 20.21 & 3.03 &
22.35 & 2.75 \\ \hline BIC$_{0.5}$ & 6.37 & 0.27 & 8.82 & 0.00 & 11.00
& 0.00 & 13.33 & 0.00 & 14.77 & 0.24 & 16.60 & 0.00 \\ \hline
Cross-val. & 6.89 & 23.54 & 13.67 & 36.67 & 19.68 & 46.67 &
30.30 & 50.80 & 37.44 & 56.32 & 38.16 & 59.48 \\ \hline Stability
sel. & 3.11 & 0.56 & 6.56 & 2.35 & 8.56 & 2.01 & 10.96 & 2.95 &
12.05 & 4.05 & 13.65 & 4.07 \\\hline
        \end{tabular}\end{center}}
        \label{table:SpamData}
        \end{table}

\begin{figure}[t]
\begin{center}
\includegraphics[width=16cm]{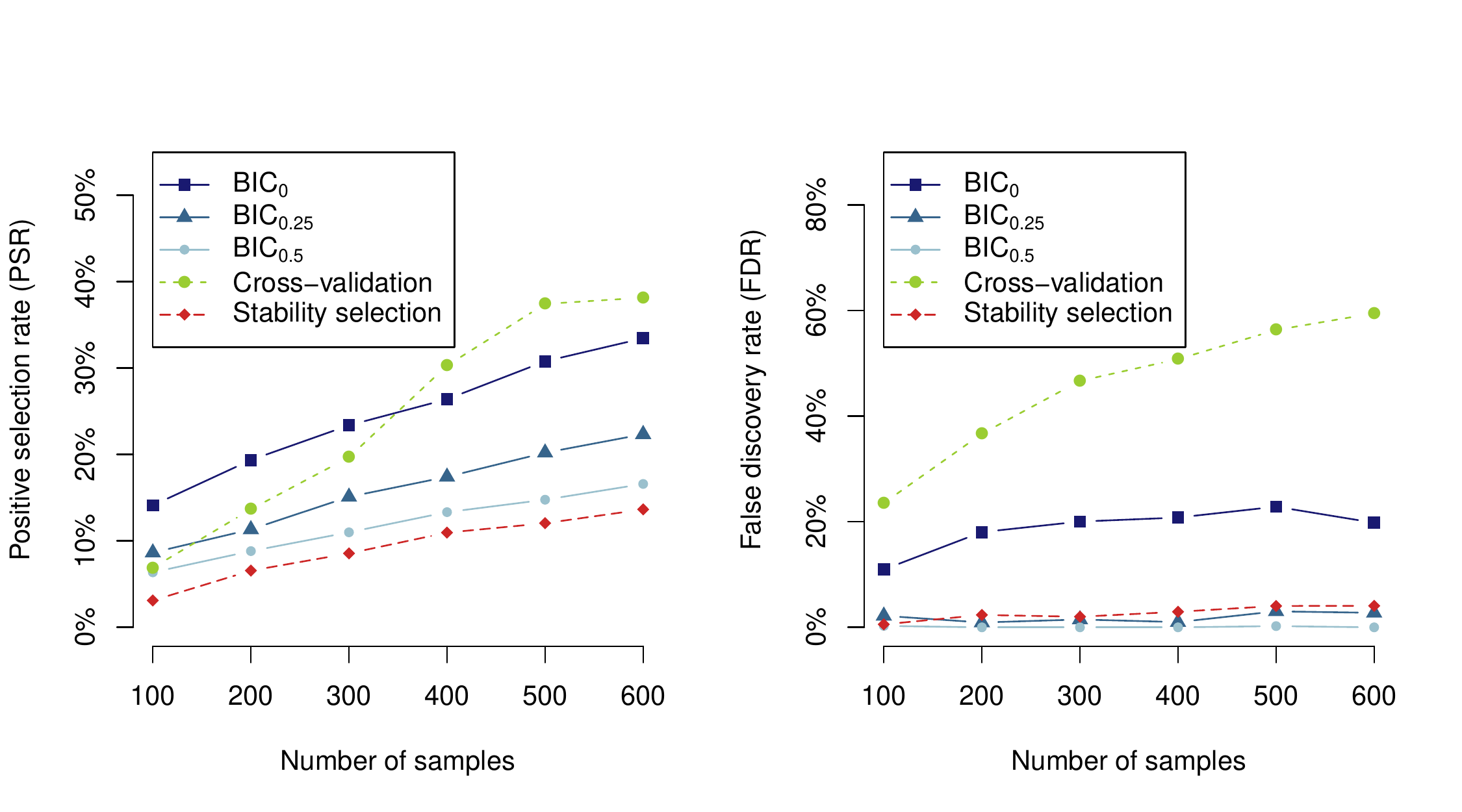}
\caption{Results for the spam email detection experiment.}
\label{fig:SpamResults}
\end{center}
\end{figure}

Comparing the three BICs to cross-validation, we observe that
cross-validation can recover more true features (for larger values of
$n$), but at an unacceptably large increase in the FDR.   
The original BIC performs better but still exhibits a high FDR.  In
contrast, the FDR of the EBIC with either $\gamma=0.25$ or
$\gamma=0.5$ remains very low at all sample sizes; the associated PSR
is smaller but increasing with the sample size.
Stability selection performed similarly to the EBIC
with $\gamma=0.5$ in this experiment, but with slightly lower PSR and
slightly higher FDR. Overall, it seems that the EBIC with
$\gamma=0.25$ performed best at the task of identifying the 57 true
features, with a very low FDR and a moderately good PSR.

The rather low PSRs observed in the simulations are due in part to the
fact that the 57 true features are not necessarily all strongly
relevant to the response.  To account for this in our evaluation of
the five methods, we ran a logistic regression using the full data set
(with a sample size of 4,601 emails) using the \texttt{glm} function
in \textsc{R}, and extracted the p-values for each feature.  For each
method, using the models selected by the method over 100 repetitions
of the experiment with $(n,p)=(600,57\cdot 24)$, we use Gaussian
smoothing (scale: standard deviation = 0.1, on the p-value scale) to
estimate, as a function of $t$, the probability that the method will
select a true feature with p-value $t$.
The estimated functions are plotted in
Figure~\ref{fig:SpamResultsPvals}. (The rate of selection of false
(permuted) features is not shown in this figure.) We see that the
function estimates for cross-validation and the BICs each decay
steadily with p-value, which seems desirable.  In this experiment,
stability selection appears to distinguish less clearly between 
highly and moderately relevant features, if we accept the p-values as
a reasonable measure of relevance.

\begin{figure}[t]
\begin{center}
\includegraphics[width=16cm]{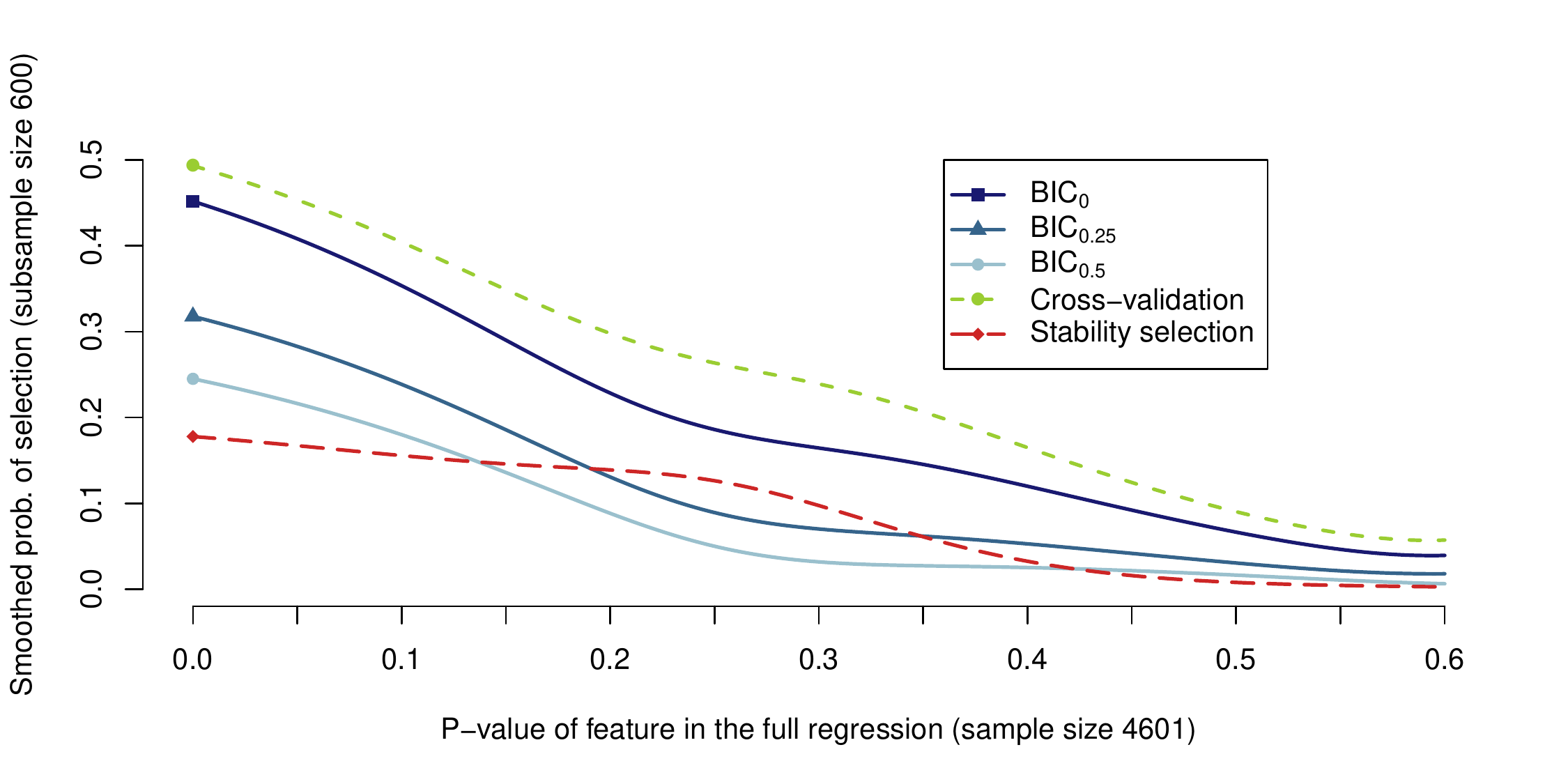}
\caption{Smoothed probability of selecting a true feature, as a function of the p-value of that feature in the full regression.}
\label{fig:SpamResultsPvals}
\end{center}
\end{figure}

\section{Edge selection in sparse graphical models}\label{sec:graphs}

In many applications, sparse graphical models are used to analyze data
arising from multivariate observations with sparse dependency
structure. In the setting we treat here, an undirected graph
$G$ consists of a set of nodes $V$ representing the observed
variables, and a set of undirected edges $E\in V\times V$ representing
possible conditional dependencies between pairs of nodes.
Specifically, if two of the variables do not have an edge between
their corresponding nodes, then they are conditionally independent
given all other observed variables.  The problem of graphical model
selection consists in selecting an appropriate set of edges to include
in the graph that represents the dependency structure among the
observed variables.

In Section~\ref{sec:sparse_graphs}, we introduce different approaches
to this edge selection problem.  In Sections~\ref{sec:GGM}
and~\ref{sec:Ising}, we discuss existing and new theoretical results
for two commonly used classes of sparse graphical
models.

\subsection{Sparse graphical models}\label{sec:sparse_graphs} 

Suppose we observe $n$ independent and identically distributed random
vectors in $\mathbb{R}^p$, denoted $X_{i\bul}=(X_{i1},\dots,X_{ip})$,
for $i=1,\dots,n$.  For each graph $G$ on the set of nodes
$V=\{1,\dots,p\}$, associate the model $\mathcal{M}_G$ comprising all
distributions for which the conditional independence constraints
implied by $G$ are satisfied.  We are then interested in the recovery
of the graph $G^*$ that encodes the dependency structure in the common
true distribution of $X_{1\bul},\dots,X_{n\bul}$.

Since optimizing over the set of all (sparse) graphs is
computationally infeasible, $\ell_1$-norm penalization methods have been
considered.  These `graphical Lasso' procedures maximize the sum of
the log-likelihood function and the absolute values of the relevant
interaction parameters.  As in the regression problem in
(\ref{eq:lasso}), a tuning parameter $\rho$ is introduced to allow for
the necessary trade-off between log-likelihood function and penalty
term.  This approach is the most tractable for the Gaussian case in
which the penalty is the sum of the absolute values of the
off-diagonal entries of the precision matrix $\Theta$
\citep{Banerjee:2008,friedmanhastie}.  With the $\ell_1$-norm
promoting sparsity in the estimate $\widehat{\Theta}^{\rho}$, a graph
estimate $\widehat{G}^{\rho}_{\mathrm{glasso}}$ can be obtained by
including an edge between nodes $j$ and $k$ whenever
$\widehat{\Theta}^{\rho}_{jk}\neq 0$.  \cite{Ravikumar:2011} show
that, under eigenvalue and irrepresentability assumptions on the true
precision matrix $\Theta^*$, the estimate
$\widehat{G}^{\rho}_{\mathrm{glasso}}$ is asymptotically consistent
for a suitable sequence of values of $\rho$.

A similar approach is the neighborhood selection method of
\cite{Meinshausen:2006}, which performs penalized regression for
selecting each node's neighborhood. Specifically, for each variable
$j$, we optimize a penalized conditional likelihood function to find
\[
\widehat{\beta}_j^{\rho}=\arg\min{}_{\b_j}\Big\{-\sum_i \log
  \mathbb{P} (X_{ij} | \{X_{ik}:k\neq j\},\beta_j)+
  \rho\|\beta_j\|_1\Big\}\;.
\]
We then define the graph estimate
$\widehat{G}^{\rho}_{\mathrm{neighbor}}$ to have an edge between nodes
$j$ and $k$ whenever $\widehat{\beta}^{\rho}_{jk}$ and
$\widehat{\beta}^{\rho}_{kj}$ are both nonzero (the \textsc{and} rule), or
whenever either $\widehat{\beta}^{\rho}_{jk}$ or
$\widehat{\beta}^{\rho}_{kj}$ is nonzero (the \textsc{or} rule).  This method
inherits
asymptotic consistency properties from results for the individual
regressions.

Both of the above methods require choosing the tuning parameter
$\rho$.  Similarly, greedy search over all graphs requires a choice of
a sparsity bound $q$, or alternately, a stopping criterion to indicate
when enough edges have been added.  In each case, we can rephrase the
tuning problem as the question of selecting a model from a small list
of candidate graphs $G_1,\dots, G_m$, of various sparsity levels.

We can use cross-validation to select a model from this list, but
there are two disadvantages.  First, $K$-fold cross-validation can be
computationally expensive due to the process of fitting models to $K$
different parts of the data.  More importantly, from the point of view
of graph recovery, cross-validation tends to choose overly large
models leading to selection of many false positive edges, in the
high-dimensional setting when $p\gg n$; compare \citet{Foygel:2010}.
As for regression, we can alternatively use stability selection
\citep{Meinshausen:2010}, where we search for edges that are stable
across sparse models fitted to subsamples of the data using graphical Lasso
or neighborhood selection; see also \cite{Liu:2010}.  This method has been shown to be
asymptotically consistent in a range of settings.  However, it again
requires refitting the model many times for different subsamples.
Finally, as a third approach, we may apply information criteria, and we
now turn to two specific settings where the extended BIC yields a
computationally inexpensive and asymptotically consistent procedure
for edge selection.

\subsection{Gaussian graphical models}\label{sec:GGM} 

Suppose the i.i.d.~observations $X_{1\bul},\dots,X_{p\bul}$ are
multivariate normal with precision (or inverse covariance) matrix
$\Theta$.  Then it is well known that $X_{1j}$ and $X_{1k}$ are
conditionally independent given the remaining variables $\{X_l:l\neq
j,k\}$ if and only if $\Theta_{jk}=0$.  The Gaussian graphical model
$\mathcal{M}_G$ associated with an undirected graph $G$ on nodes
$V=\{1,\dots,p\}$ is the set of all multivariate normal
distributions with $\Theta_{jk}=0$ when $j$ and $k$ are two distinct
non-adjacent nodes in $G$.

Prior work proposes the use of the extended BIC for sparse Gaussian
graphical model selection \citep{Foygel:2010,Gao:2011}.  Accounting
for a matrix parameter, the EBIC is defined as
\[
\mathrm{BIC}_{\g}(G) = -2\ell_{[n]}(\widehat\Theta_G) + |G|\cdot \log(n) +
4|G|\gamma\cdot \log(p),
\]
where $\ell_{[n]}(\widehat\Theta_G)$ denotes the maximized log-likelihood
function for the set of $n$ observations, and $|G|$ is the number of
edges in the graph.  Since each model is only fitted once (to the full data
set), this method carries relatively low computational cost, while
enjoying consistency properties.  We now state a version of the main
theorem from \cite{Foygel:2010}, which gives conditions under which
minimization of the EBIC leads to selection of the smallest true model
$G^*$ when applied to any list of sparse decomposable graphs
containing $G^*$; for a definition of decomposable graphs we refer the
reader to \cite{Lauritzen:1996}.

\begin{theorem} 
  Suppose that the true graph $G^*$ is decomposable with $|G^*|\le q$,
  and that the true precision matrix $\Theta^*\in\mathcal{M}_{G^*}$
  has bounded condition number and minimum nonzero value $\theta_0$
  bounded away from zero. Suppose that $p\propto n^{\kappa}$ for some
  $\kappa<1$, and that the true neighborhood size is bounded for each
  node. Fix any $\g>1-\frac{1}{4\kappa}$.  Then with probability
  tending to one as $n\rightarrow\infty$,
  \[
  \mathrm{BIC}_{\g}(G^*)<\min\left\{\mathrm{BIC}_{\g}(G) \ : \ 
    \text{$G$ is decomposable with $|G|\le q$}\right\}.
  \]
\end{theorem}

Together with consistency results on the graphical Lasso and on
neighborhood selection, this result implies that combining EBIC and
either graphical Lasso or neighborhood selection gives a consistent
method for edge selection under the assumptions stated.  While our
proof of the theorem relies on exact distribution theory applicable to
decomposable graphs, we conjecture that the stated result holds
without the restriction to decomposable graphs.  

\cite{Gao:2011} propose EBIC-based tuning of the so-called SCAD
penalization method for graphical model selection and give a
consistency result taylored to this method.  The version of the EBIC
studied by these authors has the maximum likelihood estimator replaced
by the SCAD estimator, and the model search is restricted to a subset
of the SCAD regularization path.  No decomposability assumptions were
needed by \cite{Gao:2011}.

\subsection{Ising models}
\label{sec:Ising} 

In the setting of binary observations
$X_{1\bul},\dots,X_{n\bul}\in\{0, 1\}^p$, the Ising model consists of
probability mass functions of the form
\begin{equation}\label{eq:Ising}
  \PP\big((X_{11},\dots,X_{1p})=(x_1,\dots,x_p)\big)\propto\exp\bigg\{\sum_j
  \zeta_j x_j+ \tfrac{1}{2}\sum_{j\neq k}  \Theta_{jk} x_j
  x_k\bigg\}\;,
\end{equation}
where $\zeta\in\R^p$ is any vector, and for identifiability we
constrain $\Theta\in\R^{p\times p}$ to be a symmetric matrix with zero
diagonal.  This model originated in physics to model states of
particles, where informally we have $\Theta_{jk}>0$ if
particles $j$ and $k$ prefer to be in the same state, and
$\Theta_{jk}<0$ if particles $j$ and $k$ prefer to be in
different states. For background and applications, compare e.g.\ 
\citet{Kindermann:1980}.  

In the Ising model, the conditional distribution of $X_{1j}$ given
$\{X_{1k}:k\neq j\}$ comes from the logistic
model---from~(\ref{eq:Ising}), we obtain
\[
\PP\big(X_{1j}=x_j|\{X_{1k}=x_k:k\neq j\}\big)\propto
\exp\bigg\{\bigg(\zeta_j+\sum_{k\neq j}\Theta_{jk} x_k\bigg) x_j\bigg\}\;,
\]
and therefore the log-odds are
\[
\log\left(\frac{\PP\big(X_{1j}=1|\{X_{1k}:k\neq
    j\}\big)}{\PP\big(X_{1j}=0|\{X_{1k}:k\neq j\}\big)}\right)=
\zeta_j+\sum_{k\neq
  j}\Theta_{jk} X_{1k}\;.
\]
To recover the true graph $G^*$ that describes the dependencies among
the variables (or equivalently, the sparsity pattern in the true
matrix $\Theta^*$), we can thus use neighborhood selection with the
logistic Lasso, which finds
\begin{align*}
\widehat{\beta}_j^{\rho}&=\arg\min{}_{\b_j}\bigg\{-\sum_i \log \PP\big(X_{ij} | \{X_{ik}:k\neq j\},\beta_j\big)+ \rho\|\beta_j\|_1\bigg\}\\
&=\arg\min{}_{\b_j}\bigg\{-X_{ij} \cdot \bigg(\b_{j0}+\sum_{k\neq j}X_{ik}\b_{jk}\bigg) + \log\bigg(1+\exp\Big\{\b_{j0}+\sum_{k\neq j}X_{ik}\b_{jk}\Big\}\bigg)+ \rho\|\beta_j\|_1\bigg\}\;.
\end{align*}
The resulting graph estimate $\widehat{G}^{\rho}$ has an edge between
nodes $j$ and $k$ based on the values of $\widehat{\b}^{\rho}_{jk}$
and $\widehat{\b}^{\rho}_{kj}$, using either an \textsc{and} or an \textsc{or} rule;
compare also \cite{Hofling:2009}.

Tuning the parameter $\rho$ can be done using the EBIC for logistic
regression.  Our results for consistency of the EBIC for logistic
regression then imply consistency guarantees for neighborhood selection
with EBIC tuning.  We assume that the following conditions hold (for
constants $q$ and $c$):
\begin{itemize}
\item[(C1)] The growth of $p$ is subexponential, that is,
  $\log(p)=\mathbf{o}(n)$, with 
$\kappa\coloneqq \lim\sup\log_n(p)\in[0,\infty]$. 
\item[(C2)] The true graph $G^*$ has degree bounded by $q$, that is,
  each node $j$ has a neighborhood of cardinality $\left|\{k:(j,k)\in
    G^*\}\right|\leq q$.
\item[(C3)] The true parameters are bounded with $\max{}_j
  |\zeta^*_j|\le c$ and $\max{}_{j,k}|\Theta^*_{jk}|\leq c$.
\item[(C4)] The signal is bounded away from zero such that
  \[
  \sqrt{\frac{\log(np)}{n}}=\mathbf{o}\left(\min_{(j,k)\in
      G^*}|\Theta^*_{jk}|\right)\;. 
  \]
\end{itemize}
The following theorem gives a precise statement of the consistency
properties of
the EBIC for 
edge selection in the Ising model.

\begin{theorem}\label{thm:Ising}
  Assume that conditions~(C1)-(C4) hold. Let
  $X_{1\bul},\dots,X_{n\bul}\in\{0,1\}^p$ be i.i.d.\ draws from an
  Ising model with parameters $\zeta^*\in\R^p$ and
  $\Theta^*\in\R^{p\times p}$, where $\Theta^*$ is symmetric with zero
  diagonals. Let $G^*$ be the graph with edges indicating the nonzero
  entries of $\Theta^*$, and for each node $j$, let $\mathcal{S}^*_j$
  denote its true neighborhood, that is, $\mathcal{S}_j=\{k\neq
  j:\Theta^*_{jk}\neq 0\}$.  Choose three scalars $\a,\b,\g$ to
  satisfy
 \begin{equation*}
   \left\{\text{\begin{tabular}{ll}\T\B
         $\g>1-\tfrac{1}{2\kappa}+\b+\tfrac{\a}{\kappa}$,&if 
         $\kappa>0$,\\ 
         \T\B$\a\in\left(0,\tfrac{1}{2}\right)$ and $\b>0$,&if
         $\kappa=0$.\\\end{tabular}}\right.
 \end{equation*}
 Then, for sufficiently large $n$, the event that the inequalities
 \[
 \mathrm{BIC}_{\g}(\mathcal{S}^*_j)<\min\left\{\mathrm{BIC}_{\g}(\mathcal{S}_j)
 :  \mathcal{S}_j\not\ni j, \ \mathcal{S}_j\neq \mathcal{S}_j^*, \
 |\mathcal{S}_j|\leq
 q\right\}-\log(p)\cdot\left(\g-\left(1-\frac{1}{2\kappa}+\b+\frac{\a}{\kappa}\right)\right)
 \]
 hold simultaneously for all $j$ has probability at least
 $1-n^{-\a}p^{-(\b-1)}$. In particular, the EBIC is consistent for
 neighborhood selection (simultaneously for all nodes) in the Ising
 model, whenever $\g>2-\frac{1}{2\kappa}$.

\end{theorem}

\subsection{Experiment for the Ising model}

We compared the BIC, the EBIC with $\g=0.25$ and $\g=0.5$, 10-fold
cross-validation, and stability selection as in
\cite{Meinshausen:2010} on the task of edge selection under an Ising
model for precipitation data from weather stations across four states
in the midwest region of the U.S.: Illinois, Indiana, Iowa, and
Missouri. Performance is measured relative to the true geographical
layout of the weather stations, which is ``unknown'' to the 
procedures we compare.

\subsubsection{Data and methods for model selection}

We used data from the United States Historical Climatology Network
\citep{WeatherData}.\footnote{Available at
  \texttt{http://cdiac.ornl.gov/ftp/ushcn\_daily/}} The data consists
of weather-related variables that were recorded on a daily basis.  We
specifically gathered the precipitation data, which gives the total
amount of precipitation for each day.  Trying to limit the effects of
temporal dependencies between successive observations, we took data
from the 1st and 16th of each month.  These are then treated as
independent.
We removed weather stations where data availability was low and
discarded observations with missing values for any of the remaining
weather stations.  A total of 278 days and 89 stations remained in the
final data set. Next, we hypothesized a ``true'' graph by computing
the Delaunay triangulation of these 89 weather stations, based on
their geographic locations,
 using the \texttt{delaunay} command in
 \cite{MATLAB}. Figure~\ref{fig:MidwestMap} shows a map with the
 resulting undirected graph.

\begin{figure}[t]
\begin{center}
\includegraphics[width=5cm]{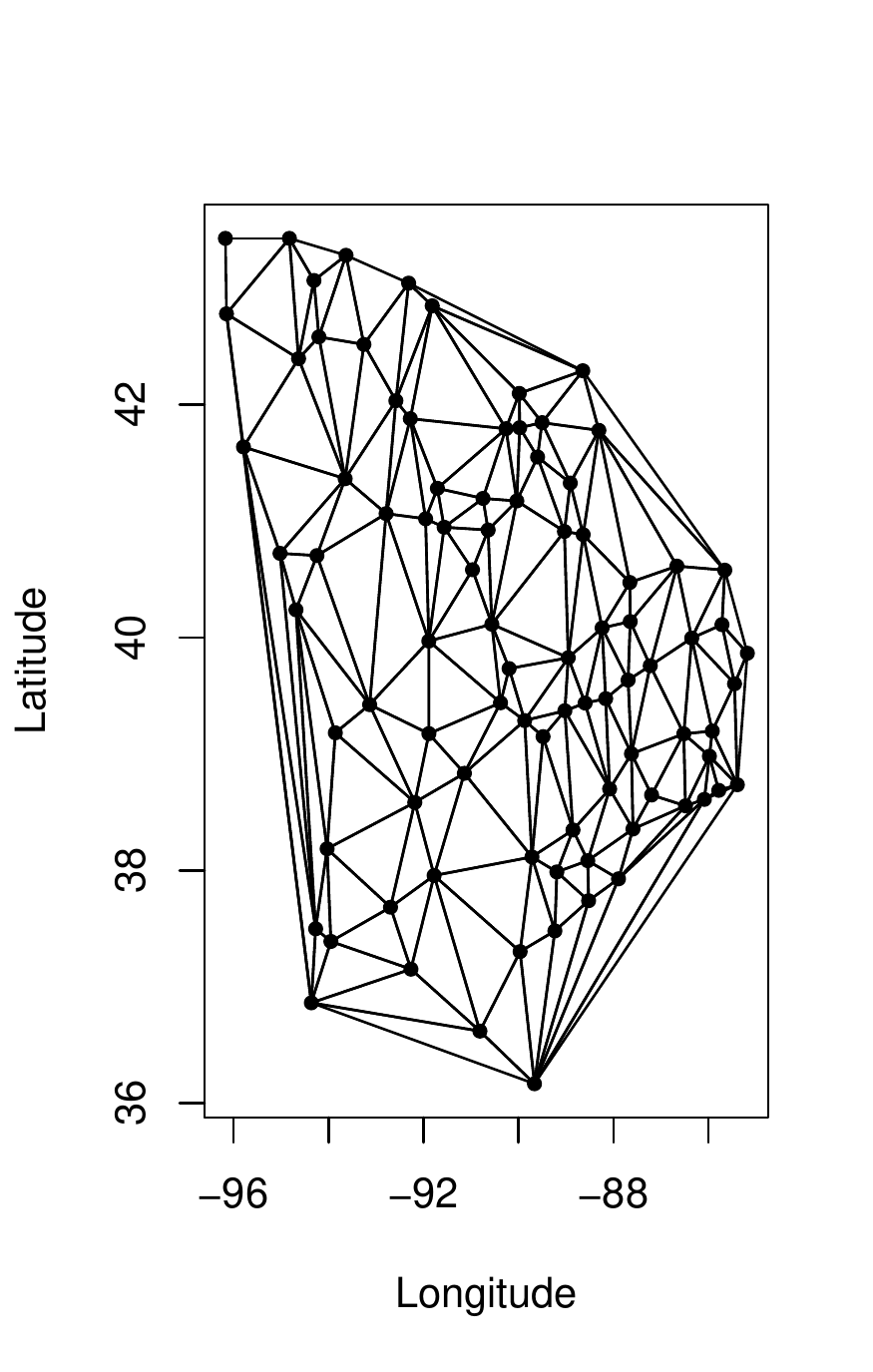}
\caption{Delaunay triangulation for 89 weather stations in Illinois,
  Indiana, Iowa, and Missouri.} 
\label{fig:MidwestMap}
\end{center}
\end{figure}

For each weather station $j$, we define binary variables $X_{ij}$
taking values 1 or 0 depending on whether or not there was a
positive amount of rainfall at weather station $j$ on day $i$.  For
each one of the stations $j$, we then applied each of the five methods
to perform a sparse logistic regression that has response vector
$X_{\bul j}$ and covariates $\{X_{\bul k}:k\neq j\}$.  Our method for
selecting a neighborhood for weather station $j$, for each of the five
methods, is identical to our methods for the regression experiment on
email data (see Section~\ref{sec:EmailMethods}).  Finally, we combined
each method with the \textsc{or} rule and with the \textsc{and} rule to produce a sparse
graph, for a total of ten methods.

\subsubsection{Results}

To evaluate the methods, we first treat the graph obtained via the
Delaunay triangulation as the ``true'' underlying graphical model.
Table~\ref{table:WeatherData} shows the results for each method,
stated in terms of positive selection rate (PSR) and false discovery
rate (FDR), relative to the ``true'' Delaunay triangulation graph.
These results are also displayed in
Figure~\ref{fig:GraphResults_PSR_FDR}, while the graphs in
Figure~\ref{fig:RecoveredGraphs} show the recovered graphs for each of
the methods, combined with the \textsc{and} or \textsc{or} rules.

\begin{table}[t]
        \caption{Positive selection rate and false discovery rate in
          the weather data experiment.} 
\begin{center}
\begin{tabular}{|@{\,}c@{\,}||r|r||r|r|}
\hline
&\multicolumn{2}{c||}{\textsc{or} rule}&\multicolumn{2}{c|}{\textsc{and} rule}\\
&PSR&FDR&PSR&FDR\\\hline\hline
BIC$_{0.0}$& 50.99 & 47.13 & 36.36 & 30.83\\ \hline
BIC$_{0.25}$ & 45.45 & 39.47& 30.04 & 28.97\\ \hline
BIC$_{0.5}$& 42.29 & 33.95 & 23.32 & 26.25\\ \hline
Cross-validation & 69.17 & 76.42 & 59.29 & 65.83 \\ \hline
Stability selection & 24.51 & 26.19 & 13.44 & 26.09\\ \hline
        \end{tabular}\end{center}
        \label{table:WeatherData}
        \end{table}
        
\begin{figure}[t]
\begin{center}
\includegraphics[width=4in]{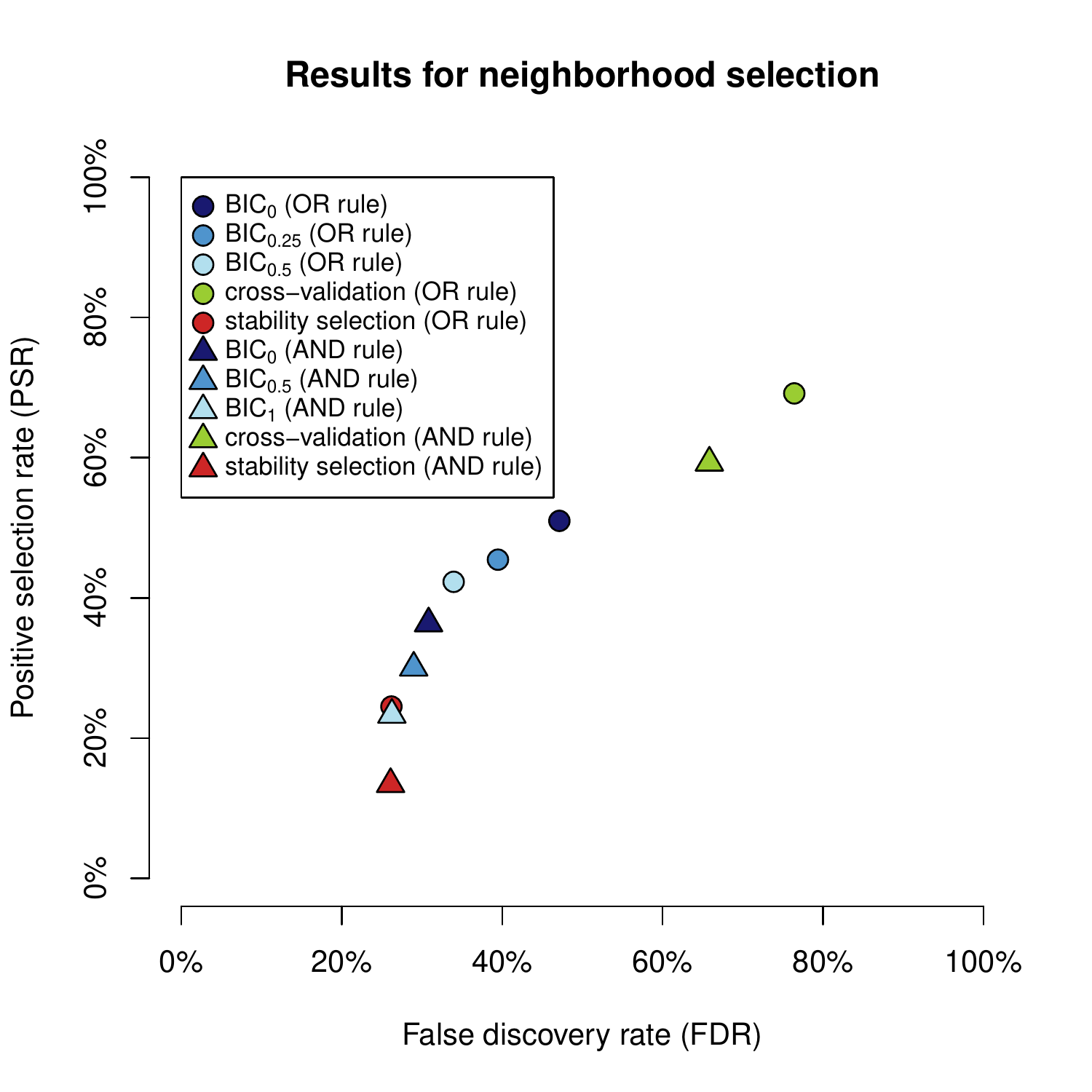}
\caption{Performance of each method, under the \textsc{or} rule and the \textsc{and}
  rule, where the true graph is defined via the Delaunay
  triangulation.} 
\label{fig:GraphResults_PSR_FDR}
\end{center}
\end{figure}

\begin{figure}[t]
\begin{center}
\includegraphics[width=15cm]{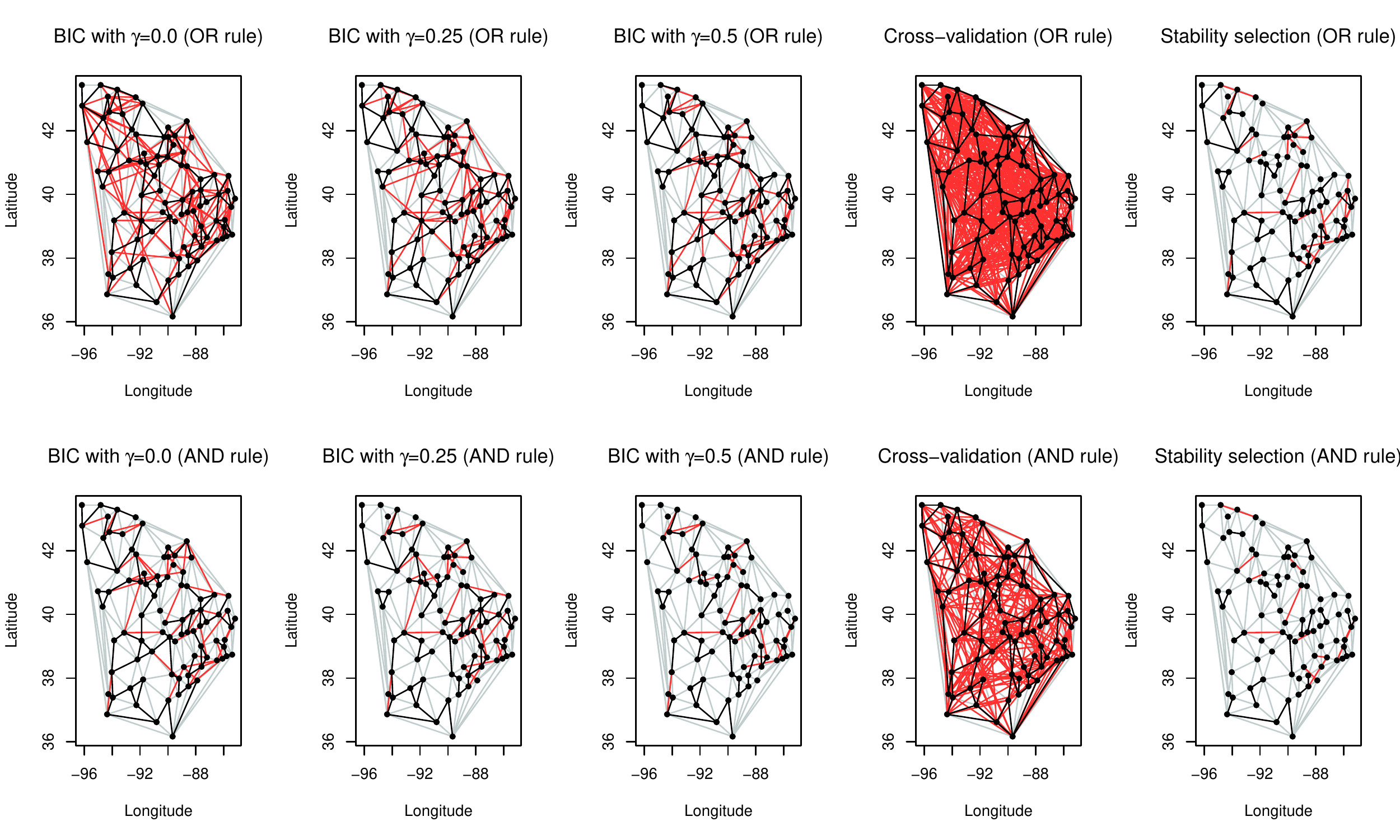}
\caption{Graphs recovered under each method. (Black edges indicate
  true positives, red edges indicate false positives, and light gray
  edges indicate false negatives, i.e.\ true edges that were not
  recovered by the method, where the true graph is defined via the
  Delaunay triangulation.)} 
\label{fig:RecoveredGraphs}
\end{center}
\end{figure}

We see that cross-validation leads to a PSR that is somewhat higher
than that of the other methods, under either an \textsc{and} or an \textsc{or} rule.
However, this comes at a drastically higher FDR.  For the EBIC, as we
increase $\gamma$, we reduce the FDR at a cost of a lower PSR, as
expected.  Stability selection appears to be a more conservative
method than $\mathrm{BIC}_{\g}$ for $\g=0.0,0.25,0.5$, with lower FDR
and lower PSR, and was substantially more computationally expensive.
While not shown, setting $\g=1.0$ with the EBIC yielded very similar
results to stability selection, in this experiment.

The edges of the Delaunay triangulation likely capture the strongest
dependencies, but it is reasonable to expect additional dependencies
that are not captured by the edges in the triangulation.  One way to
compare the methods without referring to the Delaunay triangulation is
to use the geographic distance between each pair of weather stations.
For each method, we use Gaussian smoothing (scale: standard deviation
= 10 miles) to estimate, as a function of $d$, the probability that
the method will infer an edge between two nodes that are $d$ miles
apart.  The resulting functions are plotted in
Figure~\ref{fig:GraphSmoothedProbs}, where we also show the same
smoothed function calculation for the graph defined by the Delaunay
triangulation.

We observe that the smoothed function for the cross-validation methods
(under either the \textsc{or} or the \textsc{and} rule) does
not decay to zero as distance increases.  That is, in this experiment,
the cross-validation methods tended to select some positive proportion
of edges between nodes that are arbitrarily far apart, which is undesirable. 
To a lesser extent, the same
problem occurs for the (original) BIC combined with the \textsc{or} rule. The
other methods, in contrast, yield functions that do decay to zero as
distance increases.  We see also that for two nearby weather stations,
the extended BIC with $\gamma=0.25$ or $\gamma=0.5$ combined with the
\textsc{or} rule, are both significantly more likely to select an edge than the
remaining methods, which are more conservative. Overall, the
performance of the extended BIC compares favorably to the other
methods, with a moderately good rate of edge selection for nearby
weather stations, and with probability of edge selection decaying to
zero when the distance between a pair of weather stations is large.

\begin{figure}[t]
\begin{center}
\includegraphics[width=16cm]{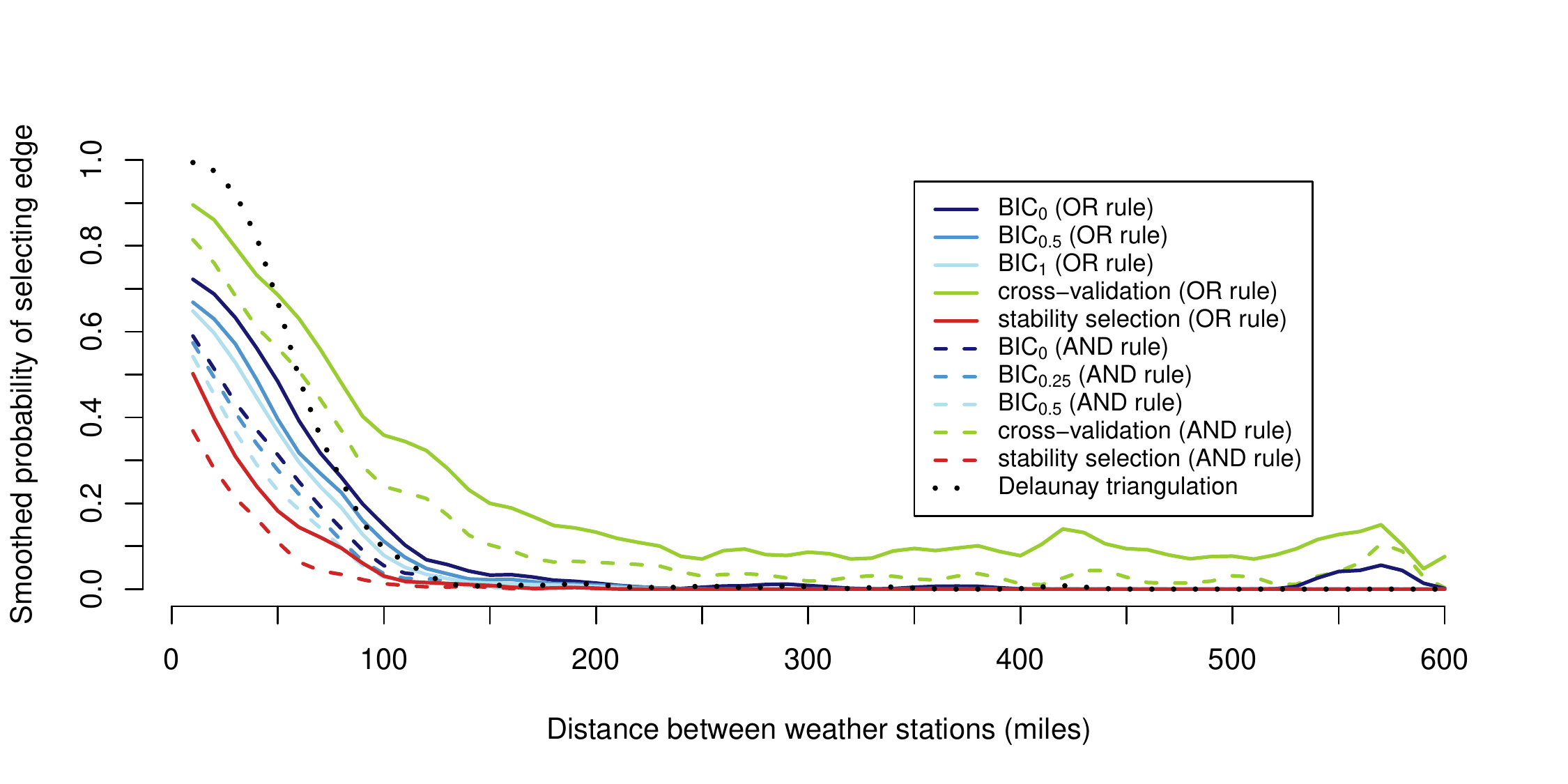}
\caption{Smoothed probability of selecting edges as a function of distance, for each method under the \textsc{or} rule and the \textsc{and} rule.}
\label{fig:GraphSmoothedProbs}
\end{center}
\end{figure}

\section{Proof sketches for theorems}\label{sec:proofs}

To prove our Theorems, we use Taylor series to approximate
log-likelihood functions, and Laplace approximations to approximate
integrated likelihoods. In Section~\ref{sec:Prelim} we introduce
notation and state two technical lemmas bounding various quantities
relating to the log-likelihood function. In
Sections~\ref{proof:Bayes=BIC}, \ref{proof:BIC_consistent},
and~\ref{proof:Ising} we outline the proofs of
Theorems~\ref{thm:Bayes=BIC}, \ref{thm:BIC_consistent},
and~\ref{thm:Ising}, respectively.  Full proofs are in the
Appendix.

\subsection{Preliminaries}\label{sec:Prelim}
Let $s_{[n]}(\phi)=\nabla \log L_{[n]}(\phi)\in\R^p$ be the gradient of the log
likelihood (or score) function, and let $H_{[n]}(\phi)=-\nabla
s_{[n]}(\phi)\in\R^{p \times p}$ be the Hessian. Write $s_J(\phi)$ and
$H_J(\phi)$ to denote the sub-vector and sub-matrix, respectively,
indexed by $j\in J$.

The following lemma gives bounds that will be important in the proofs
of the Theorems.

\begin{lemma}\label{lem:WHP}
  Fix any $\a,\b>0$.  Assume (B1)-(B5) hold, and that either (A1) or
  (A2) holds.  For sufficiently large $n$, with probability at least
  $1-n^{-\a}p^{-\b}$ under (A1), or with probability at least
  $1-n^{-\a}p^{-\b}-4K^{K+1}n^{-\frac{K-2\kappa}{2}}$ under (A2), the
  following statements are all true.  The symbols $C_1$, $C_2$,
  $\l_1^*$, $\tau$, $R$, $\l_1$, $\l_2$, and $\l_3$ appearing in the
  statements represent constants that do not depend on $n$, $p$, or on
  the data, but generally are functions of other constants appearing
  in our assumptions.
  \begin{itemize}
  \item[(i)]The gradient of the likelihood is bounded at the true
    parameter vector $\phi^*$:
    \begin{equation*}\left\|\left(H_J(\phi^*)^{-\nicefrac{1}{2}}\right)s_J(\phi^*)\right\|_{2}<\sqrt{2\left(1+\e_n\right)\left|J\backslash J^*\right|\log\left(n^{\a}p^{1+\b}\right)}\text{ for all $J\supsetneq J^*$ with $|J|\leq 2q$}\;,\end{equation*}
    where
    $\e_n=C_1\sqrt{\frac{\log\left(n^{\a}p^{1+\b}\right)}{n}}+C_2\frac{1}{\log(n)}=\mathbf{o}(1)$.
  \item[(ii)] Likelihood is upper-bounded by a quadratic function:
    \begin{equation*}\log \left(\frac{L_{[n]}(\phi^*+\psi_J)}{L_{[n]}(\phi^*)}\right)\leq -\frac{\l_1^*n}{2}\|\psi_J\|_2\left(\min\{1,\|\psi_J\|_2\}-\tau\sqrt{\frac{\log\left(n^{\a}p^{1+\b}\right)}{n}}\right)\text{ for all $|J|\leq 2q$, $\psi_J\in\R^J$}\;.\end{equation*}
  \item[(iii)]For all sparse models, the MLE lies inside a compact set:
    \begin{equation*}\big\|\phatj\big\|_2\leq R\text{ for all $|J|\leq 2q$}\;.\end{equation*}
  \item[(iv)]The eigenvalues of the Hessian are bounded from above and below, and local changes in the Hessian are bounded from above, on the relevant compact set:
    \begin{align*}&\text{For all $|J|\leq 2q$, $\|\phi_J\|_2\leq R+1$, } \lambda_1\mathbf{I}_J\preceq \tfrac{1}{n}H_J(\phi_J)\preceq \lambda_2 \mathbf{I}_J\;,\\
      &\text{and for all $\|\phi_J\|_2,\|\phi'_J\|_2\leq R+1$, } \tfrac{1}{n}\left(H_J(\phi_J)-H_J(\phi'_J)\right)\preceq \|\phi_J-\phi'_J\|_2 \lambda_3\mathbf{I}_J\;.\end{align*}
\end{itemize}
\end{lemma}

We now state a second Lemma, which relates specifically to lower-bounding the eigenvalues of the Hessian, and may be of independent interest, as it holds under much weaker assumptions than those used in our other results.
\begin{lemma}\label{lem:pos_def_H}
Fix $J$ with $|J|=2q$, and radius $R>0$. Assume
 $\lambda_{\min}\left(\EE\left[X_{1J}^{}X_{1J}^T\right]\right)\geq a_1>0$ and
 $\sup_{j\in J} \EE\left[\left|X_{1j}\right|^4\right]\leq m$.
If $n$ is sufficiently large, then with probability at least $1-e^{-\left(150\cdot \left\lceil80q^2ma_1^{-2}\right\rceil\right)^{-1}n}$, for all $\phi=\phi_J$ with $\|\phi\|_2\leq r$, 
$$H_J(\phi)\succeq n\mathbf{I}_J\cdot \frac{a_1}{4} \inf\left\{\bb''(\theta):|\theta|\leq 20q^2r\sqrt{m}\left\lceil 80q^2ma_1^{-2}\right\rceil\right\}\;.$$
\end{lemma}

\subsection{Proof outline for Theorem~\ref{thm:Bayes=BIC}}
\label{proof:Bayes=BIC}

The key bound in the proof is showing that
\begin{equation*}
\int_{\phi_J\in\R^J}L_{[n]}(\phi_J)f_J(\phi_J)\; d\phi_J
=L_{[n]}(\phatj)f_J(\phatj)\cdot\left|H_J(\phatj)\right|^{-\nicefrac{1}{2}}(2\pi)^{\nicefrac{|J|}{2}}\cdot
\left(1\pm C\sqrt{\frac{\log(np)}{n}}\right)\;,
\end{equation*}
for all $J$ with $|J|\leq q$.  To this end, we calculate the marginal
likelihood in each model $J$ by splitting the integration domain into
three regions---a small neighborhood of the MLE called
$\mathcal{N}_1$, a larger region $\mathcal{N}_2\backslash
\mathcal{N}_1$ obtained from a larger neighborhood $\mathcal{N}_2$,
and the remainder of the space, given by
$\R^J\backslash\mathcal{N}_2$.

Fix any model $J$ with $|J|\leq q$, and let $\phatj$ be the
MLE. Define the neighborhoods
\begin{align*}
&\mathcal{N}_1\coloneqq\left\{\phi: \left\|H_J(\phatj)^{\nicefrac{1}{2}}(\phi_J-\phatj)\right\|_2\leq \sqrt{4\log(np)}\right\}\;,\\
&\mathcal{N}_2\coloneqq\left\{\phi: \left\|H_J(\phatj)^{\nicefrac{1}{2}}(\phi_J-\phatj)\right\|_2\leq \sqrt{\l_1n}\right\}\;.
\end{align*}
Then the marginal likelihood is the sum of the three integrals
\begin{align*}
\text{(Int1)} &\coloneqq \int_{\phi_J\in\mathcal{N}_1}L_{[n]}(\phi_J)f_J(\phi_J)\;
d\phi_J\,,\\
\text{(Int2)} &\coloneqq \int_{\phi_J\in\mathcal{N}_2\backslash\mathcal{N}_1}L_{[n]}(\phi_J)f_J(\phi_J)\; d\phi_J\,,\\
\text{(Int3)} &\coloneqq
\int_{\phi_J\in\R^J\backslash\mathcal{N}_2}L_{[n]}(\phi_J)f_J(\phi_J)\; 
  d\phi_J\,.
\end{align*}
The bulk of the proof now  consists of computing approximations to each of
the three terms separately.  

It is at first surprising that we split into three regions, rather
than two, as is done in other work with `fixed $p$.'  The intuition
for our split is as follows:
\begin{description}
\item[\rm\em (Int1)] In the smallest region, $\mathcal{N}_1$, a
  quadratic approximation to the log-likelihood function is extremely
  accurate, and we can use it to prove the accuracy of the Laplace
  approximation for this part of the integral.
\item[\rm\em (Int2)] In the intermediate region
  $\mathcal{N}_2\backslash\mathcal{N}_1$, while the quadratic
  approximation to the log-likelihood function is no longer very
  accurate, we can still obtain a quadratic upper bound.  Hence,
  the  integrand behaves as $e^{-c\|\phi_J-\phatj\|^2_2}$
  for an appropriate constant $c$, meaning that we can use tail
  bounds for the $\chi^2$ distribution to prove that the contribution
  of this region is negligible.
\item[\rm\em (Int3)] Outside of $\mathcal{N}_2$, the quadratic
  approximation may no longer be accurate enough to use the same
  reasoning as for the intermediate region. However, due to convexity
  of log-likelihood function, the integrand is at most roughly
  $e^{-c'\|\phi_J-\phatj\|_2}$ for an appropriate constant $c'$.
  (Note that the quantity in the exponent is no longer squared.)
  Therefore, we can use tail bounds for an exponential distribution, to
  show that the contribution from this third region is also
  negligible. 
\end{description}
 
Exponential tail bounds are much weaker than those for the $\chi^2$
distribution, which explains why we separate the area outside of
$\mathcal{N}_1$ into two regions---first, the intermediate region
$\mathcal{N}_2\backslash\mathcal{N}_1$ which contains points that are
relatively close to the MLE, for which we can apply strong tail
bounds, and second, a region $\R^J\backslash\mathcal{N}_2$ where we can
only apply weaker tail bounds, but where all  points are rather
far from the MLE.

\subsection{Proof outline for
  Theorem~\ref{thm:BIC_consistent}}\label{proof:BIC_consistent} 

Lemma~\ref{lem:WHP} deals with issues arising from random covariates.
Given the results of Lemma~\ref{lem:WHP}, our proof of this theorem
follows the same reasoning as that of \citet{Chen:2011}.  Only slight
modifications are needed; we give the details in the appendix for
completeness.  The proof consists of two parts that separate the
treatment of incorrect and of true models:

\begin{itemize}
\item[(a)] An incorrect sparse model is a model $J$ with $J\not\supset
  J^*$.  In such a model, the distance between $\phatj$ and $\phi^*$
  will be large enough such that the likelihood function of model $J$
  achieves only low values.  The model will thus not be chosen over
  the true model $J^*$. Specifically, the lower-bound on the signal in
  assumption (B5) ensures that the change in the EBIC when comparing
  model $J$ to model $J^*$, is at least on the order of $\log(np)$.
\item[(b)] A true model is a model $J$ with $J\supsetneq J^*$.  In an
  overly-large true model, the achievable increase in likelihood due
  to the extra degrees of freedom will not be large enough to
  compensate for the increased model size, and so again $J$ will not
  be chosen over the smallest true model $J^*$.  Specifically, the
  increase in the achievable log-likelihood will be bounded on the
  order of $|J\backslash J^*| \log(np)$, which will be outweighed by
  the additional penalty on the larger model $J$.
\end{itemize}

\subsection{Proof outline for
  Theorem~\ref{thm:Ising}}\label{proof:Ising}

Considering each of the $p$ regressions separately, we obtain
consistency of the EBIC with probability at least $1-n^{-\a}p^{-\b}$
via Theorem~\ref{thm:BIC_consistent}, as long as all the
conditions~(B1)-(B5) hold.  Using our assumptions for this current
theorem, all these conditions hold by assumption, except for the
eigenvalue bounds on $\EE\left[X_{1J}^{}X_{1J}^T\right]$ for all
$|J|\leq 2q$.  We derive these bounds in the appendix, using
properties of the logistic model combined with the conditions assumed
to be true.

\section{Conclusion}\label{sec:conclusion}

As discussed in detail in the introduction in
Section~\ref{sec:introduction}, the results in this paper make a
formal connection between Bayesian model determination and model
search using recently-proposed extended Bayesian information criteria
(EBIC).  Our results pertain to sparse high-dimensional generalized
linear models based on a one-dimensional univariate exponential family
and with canonical link.  Evidently, a number of generalizations would
be of interest for future work.

Remaining in the univarate exponential family framework, regression
under non-canonical link could be considered in a fashion similar to
what we have done here.  Very recently, a treatment of this problem in
the vein of \cite{Chen:2011} has been undertaken by
\cite{Luo:2011:new}.  Another extension would be to allow for
exponential families with more than one parameter in regression
models.  This would in particular recover results for linear
regression with unknown variance as a special case.

A different paradigm would be the graphical model setting.  As
reviewed in Section~\ref{sec:sparse_graphs}, there is a version of the
EBIC that enjoys consistency properties in the Gaussian case.
However, it remains an open problem to establish a formal connection
to fully Bayesian graph selection procedures.  Moreover, we hope that
the available consistency results can be strengthened to avoid, in
particular, decomposability assumptions for the concerned graphs.

\section*{Acknowledgments}

Mathias Drton was supported by the NSF under Grant No.~DMS-0746265 and
by an Alfred P. Sloan Fellowship.

\bibliographystyle{chicago} 
\bibliography{BICvsBayes}

\clearpage
\appendix

\section{Proof of Theorem~\ref{thm:Bayes=BIC}}\label{appendix:Bayes=BIC}
 \begin{reptheorem}{thm:Bayes=BIC}
 Assume that conditions~(B1)-(B5) hold, and that either assumption
  (A1) or (A2) holds.  Moreover, assume the following mild conditions
  on the family of priors $\left(f_J:J\subset[p],|J|\leq q\right)$,
  which require the existence of constants $0<F_1,F_2,F_3<\infty$
  such that, uniformly for all $|J|\leq q$, we have
  \begin{itemize}
  \item[(i)] an upper bound on the priors: 
    \[
    \sup{}_{\phi_J} f_J(\phi_J)\leq F_1<\infty,
    \]
  \item[(ii)] a lower bound on the priors over a compact set:
    \[
    \inf{}_{\|\phi_J\|_2\leq R+1} f_J(\phi_J)\geq F_2>0,
    \]
    where $R$ is a function of the constants in assumptions (A1) or (A2) and (B1)-(B5), defined in the proofs,
  \item[(iii)] a Lipschitz property on the same compact set:
    \[
    \sup{}_{\|\phi_J\|_2\leq R+1} \left\|\nabla
      f_J(\phi_J)\right\|_2\leq F_3<\infty.
    \]
  \end{itemize}
  Then there is a constant $C$, no larger than
  $4F_3F_2^{-1}\l_1^{-\nicefrac{1}{2}}+2q\l_3\l_1^{-\nicefrac{3}{2}}+2$, 
  such that, for sufficiently large $n$, the event that 
  \begin{equation}\label{eq:Bayes=Laplace}
    \mathrm{Bayes}(J) 
    =P(J)\cdot L_{[n]}(\phatj)f_J(\phatj)\cdot
    \left|H_J(\phatj)\right|^{-\nicefrac{1}{2}}(2\pi)^{\nicefrac{|J|}{2}}\cdot
    \left(1\pm C\sqrt{\frac{\log(np)}{n}}\right)
  \end{equation}
  uniformly for all models $J$ with $|J|\leq q$ occurs with
  probability at least $1-(np)^{-1}$ under (A1), and with probability at
  least $1-(np)^{-1}-4K^{K+1}n^{-\frac{K-2\kappa}{2}}$ under (A2).
  In particular, for the (unnormalized) prior $P(J)= {p\choose
    |J|}^{-\gamma}\cdot\one{|J|\leq q}$, it holds that 
  \begin{equation}\label{eq:Bayes=BIC}
    \left|\log\left(\mathrm{Bayes}_{\g}(J)\right)-
      \left(-\tfrac{1}{2}\mathrm{BIC}_{\g}(J)\right)\right|\leq
    C_1\;,
  \end{equation}
  where $C_1$ is a constant no larger than
  $\tfrac{q}{2}\log(2\pi)+\g
    q\log(2q)+q\log\max\{\l_1^{-1},\l_2\}+\log\max\{F_1,F_2^{-1}\}+1$.
 \end{reptheorem}

\begin{proof}
First, we show that the approximation~(\ref{eq:Bayes=Laplace}) to the Bayesian marginal likelihood will imply the bound~(\ref{eq:Bayes=BIC}). We have
\begin{align*}
&\log\left(\mathrm{Bayes}_{\g}(J)\right)=\log\left(P(J)\cdot L_{[n]}(\phatj)f_J(\phatj)\cdot\left|H_J(\phatj)\right|^{-\nicefrac{1}{2}}(2\pi)^{\nicefrac{|J|}{2}}\cdot \left(1\pm C\sqrt{\frac{\log(np)}{n}}\right)\right)\\
&=-\g\log{p\choose |J|}+\log L_{[n]}(\phatj)+\log f_J(\phatj)-\tfrac{1}{2}\log \left|H_J(\phatj)\right|+\tfrac{|J|}{2}\log (2\pi)+\log  \left(1\pm C\sqrt{\frac{\log(np)}{n}}\right)\;.
\end{align*}
We now approximate some of the above terms. First, by definition of ${p\choose |J|}$, we have
$$|J|\log(p)\geq\log{p\choose |J|}\geq\log\left(\frac{(p-|J|)^{|J|}}{|J|^{|J|}}\right)\geq \log\left(\left(\frac{p}{2|J|}\right)^{|J|}\right)\geq |J|\log(p)-q\log(2q)\;.$$
Next, since $\lambda_1 n\mathbf{I}_J\preceq H_J(\phatj)\preceq \lambda_2n\mathbf{I}_J$, we have
$$|J|\log(n)+|J|\log(\l_1)=\log\left|\lambda_1n\mathbf{I}_J\right|\leq \log \left|H_J(\phatj)\right|\leq \log\left|\lambda_2n\mathbf{I}_J\right|= |J|\log(n)+|J|\log(\l_2)\;.$$
Finally, $F_2\leq \log f_J(\phatj)\leq F_1$ by assumption, and for sufficiently large $n$, $C\sqrt{\frac{\log(np)}{n}}\leq \frac{1}{2}$. Combining all of the above, we get
$$\log\left(\mathrm{Bayes}_{\g}(J)\right)= \log L_{[n]}(\phatj)-\tfrac{|J|}{2}\log(n)-\g|J|\log(p)\pm C_1=-\tfrac{1}{2}\mathrm{BIC}_{\g}(J)\pm C_1\;,$$
where we define
$$C_1\coloneqq \tfrac{q}{2}\log(2\pi)+\g q\log(2q)+q\log\max\{\l_1^{-1},\l_2\}+\log\max\{F_1,F_2^{-1}\}+1\;.$$

We next prove the approximation~(\ref{eq:Bayes=Laplace}). We need to show that
\begin{equation*}\int_{\phi_J\in\R^J}L_{[n]}(\phi_J)f_J(\phi_J)\; d\phi_J
=L_{[n]}(\phatj)f_J(\phatj)\cdot\left|H_J(\phatj)\right|^{-\nicefrac{1}{2}}(2\pi)^{\nicefrac{|J|}{2}}\cdot \left(1\pm C\sqrt{\frac{\log(np)}{n}}\right)\;,\end{equation*}
for all $J$ with $|J|\leq q$.

To this end, for each model $J$, we split the integration domain into three regions---a small neighborhood of the MLE denoted by $\mathcal{N}_1$, a larger region $\mathcal{N}_2\backslash \mathcal{N}_1$ obtained by taking a larger neighborhood $\mathcal{N}_2$ and subtracting the first region, and the remainder of the space, given by $\R^J\backslash\mathcal{N}_2$.

Fix any model $J$ with $|J|\leq q$, and let $\phatj$ be the MLE. Define the neighborhoods
\begin{align*}
&\mathcal{N}_1\coloneqq\left\{\phi: \left\|H_J(\phatj)^{\nicefrac{1}{2}}(\phi_J-\phatj)\right\|_2\leq \sqrt{4\log(np)}\right\}\;,\\
&\mathcal{N}_2\coloneqq\left\{\phi: \left\|H_J(\phatj)^{\nicefrac{1}{2}}(\phi_J-\phatj)\right\|_2\leq \sqrt{\l_1n}\right\}\;.
\end{align*}
(We assume $n$ is large so that $\mathcal{N}_1\subset \mathcal{N}_2$.) We write
\begin{align*}
&\int_{\phi_J\in\R^J}L_{[n]}(\phi_J)f_J(\phi_J)\; d\phi_J=\\
&\hspace{1cm}\underbrace{\int_{\phi_J\in\mathcal{N}_1}L_{[n]}(\phi_J)f_J(\phi_J)\; d\phi_J}_{\text{(Int1)}}
+\underbrace{\int_{\phi_J\in\mathcal{N}_2\backslash\mathcal{N}_1}L_{[n]}(\phi_J)f_J(\phi_J)\; d\phi_J}_{\text{(Int2)}}
+\underbrace{\int_{\phi_J\in\R^J\backslash\mathcal{N}_2}L_{[n]}(\phi_J)f_J(\phi_J)\; d\phi_J}_{\text{(Int3)}}\;.\end{align*}
We now approximate to each of the three terms separately. First, by Lemma~\ref{lem:WHP}(iv), for a point $\phi_J\in \mathcal{N}_2$, $\|H_J(\phatj)^{\nicefrac{1}{2}}(\phi_J-\phatj)\|_2\leq \sqrt{\l_1n}$ implies $\|\phi_J-\phatj\|_2\leq 1$. Therefore, integrals (Int1) and (Int2) are both computed in a neighborhood of radius $1$ around $\phatj$. The main idea for the computations below is that the contributions of (Int2) and (Int3) are negligible, while the value of (Int1) can be very closely approximated by using the second-order Taylor series expansion to the likelihood.

{\bf Approximating (Int1).}
In a very small neighborhood around $\phatj$, the quadratic approximation
\begin{align*}\sum_i \ell_i(\phi_J)&\approx \sum_i \ell_i(\phatj)+(\phi_J-\phatj)^Ts_J(\phatj)-\frac{1}{2}(\phi_J-\phatj)^TH_J(\phatj)(\phi_J-\phatj)\\
&= \sum_i \ell_i(\phatj)-\frac{1}{2}(\phi_J-\phatj)^TH_J(\phatj)(\phi_J-\phatj)\end{align*}
is very accurate, and we can therefore use a Laplace approximation to the integral in this small neighborhood, to obtain
\begin{align*}\text{(Int1)}&\approx \int_{\R^J}\exp\left\{\sum_i \ell_i(\phatj) -\frac{1}{2}(\phi_J-\phatj)^TH_J(\phatj)(\phi_J-\phatj)\right\}f_J(\phatj)\;d\phi_J \\
&={(2\pi)^{\nicefrac{|J|}{2}}f_J(\phatj)}{ \left|H_J(\phatj)\right|^{-\nicefrac{1}{2}}}\exp\left\{\sum_i \ell_i(\phatj)\right\}\;.\end{align*}

To make this approximation rigorous, we begin by giving precise bounds on the approximation to the likelihood in a neighborhood of $\phatj$. By Lemma~\ref{lem:WHP}(iv), for any $\phi_J$ with $\|\phi_J-\phatj\|_2\leq 1$, for some $t\in [0,1]$,  we have
\begin{align}
\notag\sum_i \left(\ell_i(\phi_J)-\ell_i(\phatj)\right)
&=\psi_J^Ts_J(\phatj)-\frac{1}{2}(\phi_J-\phatj)^TH_J(\phatj+t(\phi_J-\phatj))(\phi_J-\phatj)\\
\notag&=-\frac{1}{2}(\phi_J-\phatj)^TH_J(\phatj+t(\phi_J-\phatj))(\phi_J-\phatj)\\
\notag&= -\frac{1}{2}(\phi_J-\phatj)^TH_J(\phatj)(\phi_J-\phatj)\pm \frac{1}{2}\|\phi_J-\phatj\|^2_2\left\|H_J(\phatj+t(\phi_J-\phatj))-H_J(\phatj)\right\|_{\mathrm{sp}}\\
\label{eq:approx_logL}&=  -\frac{1}{2}(\phi_J-\phatj)^TH_J(\phatj)(\phi_J-\phatj)\pm\frac{1}{2}\|\phi_J-\phatj\|^3_2 n\lambda_3\;.
\end{align}
(Here $\|M\|_{\mathrm{sp}}$ denotes the spectral norm of the matrix $M$.)
Recall that for all $\phi_J\in\mathcal{N}_1\subset \mathcal{N}_2$, we have $\|\phi_J-\phatj\|_2\leq 1$. Applying the approximation~(\ref{eq:approx_logL})  for all $\phi_J\in\mathcal{N}_1$, we claim that
$$\text{(Int1)}= (2\pi)^{\nicefrac{|J|}{2}}f_J(\phatj) \left|H_J(\phatj)\right|^{-\nicefrac{1}{2}}\exp\left\{\sum_i \ell_i(\phatj)\right\}\cdot\left(1\pm \left(4F_3F_2^{-1}\l_1^{-\nicefrac{1}{2}}+2q\l_3\l_1^{-\nicefrac{3}{2}}+1\right)\cdot\sqrt{\frac{\log(np)}{n}}\right)\;.$$
We will now prove this bound.

Applying Lemma~\ref{lem:WHP}(iii),  $\|\phi_J\|_2\leq \|\phatj\|_2+\|\phi_J-\phatj\|_2\leq R+1$. By our assumptions on $f_J$ on the ball of radius $R+1$ at zero,
$$\left\|\frac{\partial}{\partial \phi_J}\log f_J(\phi_J)\right\|_2=\left\|\frac{\nabla f_J(\phi_J)}{f_J(\phi_J)}\right\|_2\leq F_3F_2^{-1}\;.$$
Next, since by Lemma~\ref{lem:WHP}(iv) we know that $H_J(\phatj)\succeq \l_1n\mathbf{I}_J$, we apply the definition of $\mathcal{N}_1$ to obtain
$$ \|\phi_J-\phatj\|^3_2n\l_3 \leq \sqrt{\frac{4\log(np)}{\l_1n}}\cdot \|\phi_J-\phatj\|^2_2\leq \sqrt{\frac{4\log(np)}{\l_1n}}\cdot (\phi_J-\phatj)^T H_J(\phatj)(\phi_J-\phatj)\cdot (\l_1n)^{-1}\;.$$

Applying the three above bounds, we obtain an upper bound on (Int1):
\begin{align*}
&\text{(Int1)}=\int_{\mathcal{N}_1}\exp\left\{\sum_i \ell_i(\phi_J)\right\}f_J(\phi_J)\; d\phi_J\\
&=\exp\left\{\sum_i \ell_i(\phatj)\right\}\int_{\mathcal{N}_1}\exp\left\{-\frac{1}{2}(\phi_J-\phatj)^T H_J(\phatj)(\phi_J-\phatj)\pm \frac{1}{2}\|\phi_J-\phatj\|^3_2n\l_3\right\}f_J(\phi_J)\; d\phi_J\\
&\leq \exp\left\{\sum_i \ell_i(\phatj)\right\}\int_{\mathcal{N}_1}\exp\left\{-\frac{1}{2}(\phi_J-\phatj)^T H_J(\phatj)(\phi_J-\phatj)\left(1-\sqrt{\frac{4\l_3^2\log(np)}{\l_1^3n}}\right)\right\}f_J(\phi_J)\; d\phi_J\\
&\leq \exp\left\{\sum_i \ell_i(\phatj)+\log f_J(\phatj) + \sqrt{\frac{4\log(np)}{\l_1n}}F_3F_2^{-1}\right\}\\
&\hspace{2cm}\times \int_{\mathcal{N}_1}\exp\left\{-\frac{1}{2}(\phi_J-\phatj)^T H_J(\phatj)(\phi_J-\phatj)\left(1-\sqrt{\frac{4\l_3^2\log(np)}{\l_1^3n}}\right)\right\}\; d\phi_J\;.
\intertext{Changing variables to $\xi=\left(1-\sqrt{\frac{4\l_3^2\log(np)}{\l_1^3n}}\right)^{-\nicefrac{1}{2}}H_J(\phatj)^{\nicefrac{1}{2}}(\phi_J-\phatj)$, the upper bound becomes}
&\exp\left\{\sum_i \ell_i(\phatj)+\log f_J(\phatj) + \sqrt{\frac{4\log(np)}{\l_1n}}F_3F_2^{-1}\right\}\cdot\left|H_J(\phatj)\right|^{-\nicefrac{1}{2}}\left(1-\sqrt{\frac{4\l_3^2\log(np)}{\l_1^3n}}\right)^{-\nicefrac{|J|}{2}}\\
&\hspace{2cm} \times\int_{\left\|\xi\right\|_2\leq \left(1-\sqrt{\frac{4\l_3^2\log(np)}{\l_1^3n}}\right)^{-\nicefrac{1}{2}}\sqrt{4\log(np)}}\;\exp\left\{-\frac{1}{2}\|\xi\|^2_2\right\}\; d\xi\\
&\leq \exp\left\{\sum_i \ell_i(\phatj)+\log f_J(\phatj)+ \sqrt{\frac{4\log(np)}{\l_1n}}F_3F_2^{-1}\right\}\cdot\left|H_J(\phatj)\right|^{-\nicefrac{1}{2}}\left(1-\sqrt{\frac{4\l_3^2\log(np)}{\l_1^3n}}\right)^{-\nicefrac{|J|}{2}}(2\pi)^{\nicefrac{|J|}{2}} \\
&\hspace{2cm}\times\int_{\xi\in \R^J}(2\pi)^{-\nicefrac{|J|}{2}}\exp\left\{-\frac{1}{2}\|\xi\|^2_2\right\}\; d\xi\\
&= \exp\left\{\sum_i \ell_i(\phatj)+\log f_J(\phatj)+ \sqrt{\frac{4\log(np)}{\l_1n}}F_3F_2^{-1}\right\}\cdot\left|H_J(\phatj)\right|^{-\nicefrac{1}{2}}\left(1-\sqrt{\frac{4\l_3^2\log(np)}{\l_1^3n}}\right)^{-\nicefrac{|J|}{2}}(2\pi)^{\nicefrac{|J|}{2}} \;.\\
\end{align*}

We similarly obtain a lower bound:
\begin{align*}
&\text{(Int1)}=\int_{\mathcal{N}_1}\exp\left\{\sum_i \ell_i(\phi_J)\right\}f_J(\phi_J)\; d\phi_J\\
&=\exp\left\{\sum_i \ell_i(\phatj)\right\}\int_{\mathcal{N}_1}\exp\left\{-\frac{1}{2}(\phi_J-\phatj)^T H_J(\phatj)(\phi_J-\phatj)\pm\frac{1}{2} \|\phi_J-\phatj\|^3_2n\l_3\right\}f_J(\phi_J)\; d\phi_J\\
&\geq \exp\left\{\sum_i \ell_i(\phatj)\right\}\int_{\mathcal{N}_1}\exp\left\{-\frac{1}{2}(\phi_J-\phatj)^T H_J(\phatj)(\phi_J-\phatj)\left(1+\sqrt{\frac{4\l_3^2\log(np)}{\l_1^3n}}\right)\right\}f_J(\phi_J)\; d\phi_J\\
&\geq \exp\left\{\sum_i \ell_i(\phatj)+\log f_J(\phatj) - \sqrt{\frac{4\log(np)}{\l_1n}}F_3F_2^{-1}\right\}\\
&\hspace{2cm}\times \int_{\mathcal{N}_1}\exp\left\{-\frac{1}{2}(\phi_J-\phatj)^T H_J(\phatj)(\phi_J-\phatj)\left(1+\sqrt{\frac{4\l_3^2\log(np)}{\l_1^3n}}\right)\right\}\; d\phi_J\;.
\intertext{Changing variables to $\xi=\left(1+\sqrt{\frac{4\l_3^2\log(np)}{\l_1^3n}}\right)^{-\nicefrac{1}{2}}H_J(\phatj)^{\nicefrac{1}{2}}(\phi_J-\phatj)$, the lower bound becomes}
&\exp\left\{\sum_i \ell_i(\phatj)+\log f_J(\phatj) - \sqrt{\frac{4\log(np)}{\l_1n}}F_3F_2^{-1}\right\}\cdot\left|H_J(\phatj)\right|^{-\nicefrac{1}{2}}\left(1+\sqrt{\frac{4\l_3^2\log(np)}{\l_1^3n}}\right)^{-\nicefrac{|J|}{2}}(2\pi)^{\nicefrac{|J|}{2}}\\
&\hspace{2cm} \times\int_{\left\|\xi\right\|_2\leq \left(1+\sqrt{\frac{4\l_3^2\log(np)}{\l_1^3n}}\right)^{-\nicefrac{1}{2}}\sqrt{4\log(np)}}(2\pi)^{-\nicefrac{|J|}{2}}\exp\left\{-\frac{1}{2}\|\xi\|^2_2\right\}\; d\phi_J\\
&\geq \exp\left\{\sum_i \ell_i(\phatj)+\log f_J(\phatj)- \sqrt{\frac{4\log(np)}{\l_1n}}F_3F_2^{-1}\right\}\cdot\left|H_J(\phatj)\right|^{-\nicefrac{1}{2}}\left(1+\sqrt{\frac{4\l_3^2\log(np)}{\l_1^3n}}\right)^{-\nicefrac{|J|}{2}}(2\pi)^{\nicefrac{|J|}{2}} \\
&\hspace{2cm}\times \PP\left\{\chi^2_{|J|}\leq 2\log(np)\right\}\\
&\geq \exp\left\{\sum_i \ell_i(\phatj)+\log f_J(\phatj)- \sqrt{\frac{4\log(np)}{\l_1n}}F_3F_2^{-1}\right\}\cdot\left|H_J(\phatj)\right|^{-\nicefrac{1}{2}}\left(1+\sqrt{\frac{4\l_3^2\log(np)}{\l_1^3n}}\right)^{-\nicefrac{|J|}{2}}(2\pi)^{\nicefrac{|J|}{2}}\\
&\hspace{2cm}\times \left(1-e^{-\log(np)/2}\right)\;,\\
\end{align*}
for sufficiently large $n$.

Combining the upper and lower bounds, we therefore have 
\begin{align*}
& \text{(Int1)}= 
 \exp\left\{\sum_i \ell_i(\phatj)+\log f_J(\phatj)+ \sqrt{\frac{4\log(np)}{\l_1n}}\frac{F_3}{F_2}\right\}\\
 &\hspace{1.5in}\cdot\left|H_J(\phatj)\right|^{-\nicefrac{1}{2}}\left(1-\sqrt{\frac{4\l_3^2\log(np)}{\l_1^3n}}\right)^{-\nicefrac{|J|}{2}}(2\pi)^{\nicefrac{|J|}{2}} \cdot (1-c) \;,\\
\end{align*}
for some $c$ satisfying $0\leq c\leq e^{-\log(np)/2}$. Since $\log(np)=\mathbf{o}(n)$, we can thus write
$$\text{(Int1)}= (2\pi)^{\nicefrac{|J|}{2}}f_J(\phatj) \left|H_J(\phatj)\right|^{-\nicefrac{1}{2}}\exp\left\{\sum_i \ell_i(\phatj)\right\}\cdot\left(1\pm \left(4F_3F_2^{-1}\l_1^{-\nicefrac{1}{2}}+2q\l_3\l_1^{-\nicefrac{3}{2}}+1\right)\cdot\sqrt{\frac{\log(np)}{n}}\right)\;.$$

{\bf Bounding (Int2).}
For $\|\phi_J-\phatj\|_2\leq 1$, we can apply Lemma~\ref{lem:WHP}(iii) and (iv) to see that $H_J(\phi_J)\succeq \l_1 n\mathbf{I}_J\succeq \l_1\l_2^{-1} H_J(\phatj)$. Therefore, by the Taylor series approximation, for $\|\phi_J-\phatj\|_2\leq 1$, since $s_J(\phatj)=0$, we have
\begin{equation}\label{eq:Hessian_Int2_Int3}\sum_i \ell_i(\phi_J)\leq \sum_i \ell_i(\phatj)-\frac{\l_1\l_2^{-1}}{2}(\phi_J-\phatj)^TH_J(\phatj)(\phi_J-\phatj)\;,\end{equation}
and so
\begin{align*}
\text{(Int2)}&=\int_{\sqrt{4\log(np)}<\left\|H_J(\phatj)^{\nicefrac{1}{2}}(\phi_J-\phatj)\right\|_2\leq \sqrt{\l_1n}}\exp\left\{\sum_i \ell_i(\phi_J)\right\}f_J(\phi_J)\; d\phi_J\\
&\leq F_1\int_{\sqrt{4\log(np)}<\left\|H_J(\phatj)^{\nicefrac{1}{2}}(\phi_J-\phatj)\right\|_2}\exp\left\{\sum_i \ell_i(\phatj)-\frac{\l_1\l_2^{-1}}{2}(\phi_J-\phatj)^TH_J(\phatj)(\phi_J-\phatj)\right\}\; d\phi_J\\
&= F_1\exp\left\{\sum_i \ell_i(\phatj)\right\}\left|H_J(\phatj)\right|^{-\nicefrac{1}{2}}(\l_2\l_1^{-1})^{\nicefrac{|J|}{2}}\int_{\|\xi\|^2_2>{{2\log(np)}}}\exp\left\{-\frac{1}{2}\|\xi\|^2_2\right\}\; d\xi\\
&= F_1\exp\left\{\sum_i \ell_i(\phatj)\right\}\left|H_J(\phatj)\right|^{-\nicefrac{1}{2}}(\l_2\l_1^{-1})^{\nicefrac{|J|}{2}}\cdot (2\pi)^{\nicefrac{|J|}{2}}\PP\left\{\chi^2_{|J|}>{{2\log(np)}}\right\}\\
&\leq F_1\exp\left\{\sum_i \ell_i(\phatj)\right\}\left|H_J(\phatj)\right|^{-\nicefrac{1}{2}}(\l_2\l_1^{-1})^{\nicefrac{|J|}{2}}\cdot (2\pi)^{\nicefrac{|J|}{2}}\cdot e^{-\log(np)/2}\;,
\end{align*}
by the chi-square tail bounds derived by \citet{Cai:2002}.

{\bf Bounding (Int3).}
For all $\phi_J$ such that $\left\|H_J(\phatj)^{\nicefrac{1}{2}}(\phi_J-\phatj)\right\|_2=\sqrt{\l_1n}$, by~(\ref{eq:Hessian_Int2_Int3}), we know that
\begin{align*}\sum_i \ell_i(\phi_J)-\sum_i \ell_i(\phatj)&\leq -\frac{\l_1\l_2^{-1}}{2}(\phi_J-\phatj)^TH_J(\phatj)(\phi_J-\phatj) \\&=-\frac{\l_1\l_2^{-1}\cdot \sqrt{\l_1n}}{2}\left\|H_J(\phatj)^{\nicefrac{1}{2}}(\phi_J-\phatj)\right\|_2\;,\end{align*}
and so by convexity of likelihood, for all $\phi_J$ such that $\left\|H_J(\phatj)^{\nicefrac{1}{2}}(\phi_J-\phatj)\right\|_2>  \sqrt{\l_1n}$,
$$\sum_i \ell_i(\phi_J)-\sum_i \ell_i(\phatj)\leq -\frac{\l_1\l_2^{-1}\cdot \sqrt{\l_1n}}{2}\left\|H_J(\phatj)^{\nicefrac{1}{2}}(\phi_J-\phatj)\right\|_2\;.$$

Therefore,
\begin{align*}
\text{(Int3)}&=\int_{\left\|H_J(\phatj)^{\nicefrac{1}{2}}(\phi_J-\phatj)\right\|_2>  \sqrt{\l_1n}}\exp\left\{\sum_i \ell_i(\phi_J)\right\}f_J(\phi_J)\; d\phi_J\\
&\leq F_1\exp\left\{\sum_i\ell_i(\phatj)\right\}\int_{\left\|H_J(\phatj)^{\nicefrac{1}{2}}(\phi_J-\phatj)\right\|_2>  \sqrt{\l_1n}}\exp\left\{-\frac{\l_1\l_2^{-1}\cdot \sqrt{\l_1n}}{2}\left\|H_J(\phatj)^{\nicefrac{1}{2}}(\phi_J-\phatj)\right\|_2\right\}\; d\phi_J\;.
\intertext{Changing variables to $\xi=H_J(\phatj)^{\nicefrac{1}{2}}(\phi_J-\phatj)$, the integral is equal to}
&= F_1\exp\left\{\sum_i\ell_i(\phatj)\right\}\left|H_J(\phatj)\right|^{-\nicefrac{1}{2}}\int_{\|\xi\|_2>\sqrt{\l_1n}}\exp\left\{-\frac{\l_1\l_2^{-1}\cdot \sqrt{\l_1n}}{2}\left\|\xi\right\|_2\right\}\; d\xi\\
&\leq F_1\exp\left\{\sum_i\ell_i(\phatj)\right\}\left|H_J(\phatj)\right|^{-\nicefrac{1}{2}}\cdot  \exp\left\{-n\cdot{\l_1^2}(4q\l_2)^{-1}\right\}\;,
\end{align*}
where the last inequality is proved as follows:

\begin{align*}
&\int_{\|\xi\|_2>\sqrt{\l_1n}}\exp\left\{-\frac{\l_1\l_2^{-1}\cdot \sqrt{\l_1n}}{2}\left\|\xi\right\|_2\right\}\; d\xi\\
&=2^{|J|}\int_{\xi\in\R_+^J, \|\xi\|_2>\sqrt{\l_1n}}\exp\left\{-\frac{\l_1\l_2^{-1}\cdot \sqrt{\l_1n}}{2}\cdot |J|^{-\nicefrac{1}{2}}(\xi_1+\dots+\xi_{|J|})\right\}\; d\xi\\
&\leq 2^{|J|}\int_{\xi\in\R_+^J, \|\xi\|_{\infty}>\sqrt{\l_1n|J|^{-1}}}\exp\left\{-\frac{\l_1^{1.5}\sqrt{n}}{2\l_2\sqrt{|J|}}(\xi_1+\dots+\xi_{|J|})\right\}\; d\xi\\
&=2^{|J|}\cdot \left(\frac{\l_1^{1.5}\sqrt{n}}{2\l_2\sqrt{|J|}}\right)^{-|J|}\cdot \PP\left\{\max\{Z_1,\dots,Z_{|J|}\}>\sqrt{\l_1n|J|^{-1}}:Z_1,\dots,Z_{|J|}\stackrel{\text{iid}}{\sim}\mathrm{Exp}\left(\frac{\l_1^{1.5}\sqrt{n}}{2\l_2\sqrt{|J|}}\right)\right\}\\
&\leq 2^{|J|}\cdot |J|\cdot \left(\frac{\l_1^{1.5}\sqrt{n}}{2\l_2\sqrt{|J|}}\right)^{-|J|}\cdot  \PP\left\{\mathrm{Exp}\left(\frac{\l_1^{1.5}\sqrt{n}}{2\l_2\sqrt{|J|}}\right)>\sqrt{\l_1n|J|^{-1}}\right\}\\
&\leq \PP\left\{\mathrm{Exp}\left(\frac{\l_1^{1.5}\sqrt{n}}{2\l_2\sqrt{|J|}}\right)>\sqrt{\l_1n|J|^{-1}}\right\}
= \exp\left\{-\sqrt{\l_1n|J|^{-1}}\cdot \frac{\l_1^{1.5}\sqrt{n}}{2\l_2\sqrt{|J|}}\right\}\;.
\end{align*}

{\bf Combining the bounds.}
Applying our approximation of (Int1) and bounds on (Int2) and (Int3), we have

\begin{align*}
&\int_{\phi_J\in\R^J}\exp\left\{\sum_i \ell_i(\phi_J)\right\}f_J(\phi_J)\; d\phi_J=\text{(Int1)}+\text{(Int2)}+\text{(Int3)}\\
&=(2\pi)^{\nicefrac{|J|}{2}}f_J(\phatj) \left|H_J(\phatj)\right|^{-\nicefrac{1}{2}}\exp\left\{\sum_i \ell_i(\phatj)\right\}\cdot\left(1\pm \left(4F_3F_2^{-1}\l_1^{-\nicefrac{1}{2}}+2q\l_3\l_1^{-\nicefrac{3}{2}}+1\right)\cdot\sqrt{\frac{\log(np)}{n}}\right)\\
&\hspace{1cm} \pm F_1\exp\left\{\sum_i \ell_i(\phatj)\right\}\left|H_J(\phatj)\right|^{-\nicefrac{1}{2}}(\l_2\l_1^{-1})^{\nicefrac{|J|}{2}}\cdot (2\pi)^{\nicefrac{|J|}{2}}\cdot e^{-\log(np)/2} \\
&\hspace{1cm} \pm F_1\exp\left\{\sum_i\ell_i(\phatj)\right\}\left|H_J(\phatj)\right|^{-\nicefrac{1}{2}}\cdot  \exp\left\{-n\cdot{\l_1^2}(4q\l_2)^{-1}\right\} \\
&=(2\pi)^{\nicefrac{|J|}{2}}f_J(\phatj) \left|H_J(\phatj)\right|^{-\nicefrac{1}{2}}\exp\left\{\sum_i \ell_i(\phatj)\right\}\\
&\hspace{1cm}\cdot\left(1\pm \left(4F_3F_2^{-1}\l_1^{-\nicefrac{1}{2}}+2q\l_3\l_1^{-\nicefrac{3}{2}}+1\right)\cdot\sqrt{\frac{\log(np)}{n}}\pm F_1F_2^{-1}\left(\frac{(\l_2\l_1^{-1})^{\nicefrac{|J|}{2}}}{\sqrt{np}}+ \frac{(2\pi)^{-\nicefrac{|J|}{2}}}{e^{n\cdot{\l_1^2}(4q\l_2)^{-1}}}\right)\right)\\
&=(2\pi)^{\nicefrac{|J|}{2}}f_J(\phatj) \left|H_J(\phatj)\right|^{-\nicefrac{1}{2}}\exp\left\{\sum_i \ell_i(\phatj)\right\}\cdot\left(1\pm \left(4F_3F_2^{-1}\l_1^{-\nicefrac{1}{2}}+2q\l_3\l_1^{-\nicefrac{3}{2}}+2\right)\cdot\sqrt{\frac{\log(np)}{n}}\right)\;,
\end{align*}
for sufficiently large $n$.
\end{proof}

\section{Proof of Theorem~\ref{thm:BIC_consistent}}\label{appendix:BIC_consistent}

{\bf Incorrect models.}
Fix any $J\not\supset J^*$ with $|J|\leq q$. We first consider the loss in likelihood resulting from excluding one (or more) of the true covariates. Recall that $\sqrt{\frac{\log(np)}{n}}=\mathbf{o}\left(\min\left\{\left|\phi^*_j\right| :j\in J^*\right\}\right)$ by assumption---we use this in several inequalities below, marked with a $\star$.

We apply Lemma~\ref{lem:WHP}(ii) to the set $J'=J\cup J^*$ with $\psi_{J'}=\phatj-\phi^*$, and obtain
\begin{align*}
\log L_{[n]}(\phatj)-\log L_{[n]}(\phi^*)&=\log L_{[n]}(\phi^*+(\phatj-\phi^*))-\log L_{[n]}(\phi^*) \\
&\leq -\frac{\l_1^*}{2}n\cdot\left\|\phatj-\phi^*\right\|_2\left(\min\left\{1,\left\|\phatj-\phi^*\right\|_2\right\}-\tau\sqrt{\frac{\log\left(n^{\a}p^{1+\b}\right)}{n}}\right)\\
&\stackrel{\star}{\leq} -\frac{\l_1^*}{2}n\cdot \min_{j\in J^*}\left|\phi^*_j\right|\left(\min\left\{1,\min_{j\in J^*}\left|\phi^*_j\right|\right\}-\frac{1}{2}\min_{j\in J^*}\left|\phi^*_j\right|\right)\\
&\leq -\frac{\l_1^*n}{4} \cdot \min_{j\in J^*}\left|\phi^*_j\right|^2\;.\end{align*}
Then
\begin{align*}
\mathrm{BIC}_{\g}(J)-\mathrm{BIC}_{\g}(J^*)&=-2\log L_{[n]}(\phatj)+2\log L_{[n]}(\widehat{\phi}_{J^*})+\left(|J|-|J^*|\right)\log(n)+2\g\left(|J|-|J^*|\right)\log(p)\\
&\geq -2\log L_{[n]}(\phatj)+2\log L_{[n]}(\phi^*)+\left(|J|-|J^*|\right)\log(n)+2\g\left(|J|-|J^*|\right)\log(p)\\
&\geq  \frac{\l_1^*n}{2} \cdot \min_{j\in J^*}\left|\phi^*_j\right|^2-2q\log\left(n^{\nicefrac{1}{2}}p^{\gamma}\right)\\
&\stackrel{\star}{\geq}\frac{\l_1^*n}{2} \cdot \min_{j\in J^*}\left|\phi^*_j\right|^2-\frac{\l_1^*n}{4} \cdot \min_{j\in J^*}\left|\phi^*_j\right|^2\\
&\geq  \frac{\l_1^*n}{4} \cdot \min_{j\in J^*}\left|\phi^*_j\right|^2\\
&\stackrel{\star}{\geq} \log(p)\cdot\left(\g-\left(1-\frac{1}{2\kappa}+\b+\frac{\a}{\kappa}\right)\right)\;,\end{align*}
for sufficiently large $n$.

{\bf True models.}
Fix $J\supsetneq J^*$ with $|J|\leq q$. We first compute an upper bound on the increase in likelihood due to including additional (false) covariates. For sufficiently large $n$, we apply Lemma~\ref{lem:WHP}(ii) and (iv) and obtain, for some $t\in[0,1]$,
\begin{align*}
&\log L_{[n]}(\phatj)-\log L_{[n]}(\phi^*)\\
&= (\phatj-\phi^*)^Ts_J(\phi^*)-\frac{1}{2}(\phatj-\phi^*)^TH_J(\phi^*+t(\phatj-\phi))(\phatj-\phi^*)\\
&\leq(\phatj-\phi^*)^Ts_J(\phi^*)-\frac{1}{2}(\phatj-\phi^*)^TH_J(\phi^*)(\phatj-\phi^*)+\frac{1}{2}\|\phatj-\phi^*\|^2_2\left\|H_J(\phi^*)-H_J(\phi^*+t(\phatj-\phi^*))\right\|_{\mathrm{sp}}\\
&\leq(\phatj-\phi^*)^Ts_J(\phi^*)-\frac{1}{2}(\phatj-\phi^*)^TH_J(\phi^*)(\phatj-\phi^*)+\frac{1}{2}\|\phatj-\phi^*\|^3_2\cdot n\lambda_3\\
&\leq\left[(\phatj-\phi^*)^Ts_J(\phi^*)-\frac{1}{2}(\phatj-\phi^*)^TH_J(\phi^*)(\phatj-\phi^*)\right]+\frac{1}{2}\left(\tau\sqrt{\frac{\log\left(n^{\a}p^{1+\b}\right)}{n}}\right)^3\cdot n\l_3\\
&\leq\sup_{z\in\R^J}\left(z^Ts_J(\phi^*)-\frac{1}{2}z^TH_J(\phi^*)z\right)+\frac{\l_3\tau^3}{2}\sqrt{\frac{\log\left(n^{\a}p^{1+\b}\right)}{n}}\cdot \log\left(n^{\a}p^{1+\b}\right)\\
&=\frac{1}{2}s_J(\phi^*)^TH_J(\phi^*)^{-1}s_J(\phi^*)+ \frac{\l_3\tau^3}{2}\sqrt{\frac{\log\left(n^{\a}p^{1+\b}\right)}{n}}\cdot \log\left(n^{\a}p^{1+\b}\right)\\
&\leq \left(1+\left(C_1+\frac{\l_3\tau^3}{2}\right)\sqrt{\frac{\log\left(n^{\a}p^{1+\b}\right)}{n}}+C_2\frac{1}{\log(n)}\right)\left|J\backslash J^*\right|\log\left(n^{\a}p^{1+\b}\right)\;,
\end{align*}
where the last inequality is obtained by applying Lemma~\ref{lem:WHP}(i). Hence,
\begin{align*}
&\mathrm{BIC}_{\g}(J)-\mathrm{BIC}_{\g}(J^*)\\
&= -2\log L_{[n]}(\phatj)+2\log L_{[n]}(\phat_{J^*})+|J\backslash J^*|\log(n)+2\g|J\backslash J^*|\log(p)\\
&\geq -2\log L_{[n]}(\phatj)+2\log L_{[n]}(\phi^*)+|J\backslash J^*|\log(n)+2\g|J\backslash J^*|\log(p)\\
&\geq -2\left(1+\left(C_1+\frac{\l_3\tau^3}{2}\right)\sqrt{\frac{\log\left(n^{\a}p^{1+\b}\right)}{n}}+C_2\frac{1}{\log(n)}\right)\left|J\backslash J^*\right|\log\left(n^{\a}p^{1+\b}\right)
+2|J\backslash J^*|\log\left(n^{\nicefrac{1}{2}}p^{\g}\right)\\
&=2|J\backslash J^*|\log(p)\cdot\left(\g+\frac{1}{2\kappa_n}-\left(1+\left(C_1+\frac{\l_3\tau^3}{2}\right)\sqrt{\frac{\log\left(n^{\a}p^{1+\b}\right)}{n}}+C_2\frac{1}{\log(n)}\right)\left(1+\b+\frac{\a}{2\kappa_n}\right)\right)\\
&=2|J\backslash J^*|\log(p)\cdot\left(\g+\frac{1}{2\kappa_n}-\left(1+\mathbf{o}(1)\right)\left(1+\b+\frac{\a}{2\kappa_n}\right)\right)\\
&\geq \log(p)\cdot\left(\g+\frac{1}{2\kappa}-\left(1+\b+\frac{\a}{2\kappa}\right)\right)\;,
\end{align*}
for sufficiently large $n$.

\section{Proof of Theorem~\ref{thm:Ising}}\label{appendix:Ising}
\begin{reptheorem}{thm:Ising}
Assume that conditions~(C1)-(C4) hold. Let $X_{1\bul},\dots,X_{n\bul}\in\{0,1\}^p$ be i.i.d.\ draws from an Ising model with parameters $\zeta^*\in\R^p$ and $\Theta^*\in\R^{p\times p}$, where $\Theta^*$ is symmetric with zero diagonals. Let $G^*$ be the graph with edges indicating the nonzero entries of $\Theta^*$, and for each node $j$, let $\mathcal{S}^*_j$ denote its true neighborhood, that is, $\mathcal{S}_j=\{k\neq j:\Theta^*_{jk}\neq 0\}$.
Choose three scalars $\a,\b,\g$ to satisfy
 \begin{equation*}
   \left\{\text{\begin{tabular}{ll}\T\B
         $\g>1-\tfrac{1}{2\kappa}+\b+\tfrac{\a}{\kappa}$,&if 
         $\kappa>0$,\\ 
         \T\B$\a\in\left(0,\tfrac{1}{2}\right)$ and $\b>0$,&if
         $\kappa=0$.\\\end{tabular}}\right.
 \end{equation*}
  Then, for sufficiently large $n$, the event that the inequalities
 \[
 \mathrm{BIC}_{\g}(\mathcal{S}^*_j)<\min\left\{\mathrm{BIC}_{\g}(\mathcal{S}_j)
 :  \mathcal{S}_j\not\ni j, \ \mathcal{S}_j\neq \mathcal{S}_j^*, \
 |\mathcal{S}_j|\leq
 q\right\}-\log(p)\cdot\left(\g-\left(1-\frac{1}{2\kappa}+\b+\frac{\a}{\kappa}\right)\right)
 \]
 hold simultaneously for all $j$ has probability at least
 $1-n^{-\a}p^{-(\b-1)}$. In particular, the EBIC is consistent for
 neighborhood selection (simultaneously for all nodes) in the Ising
 model, whenever $\g>2-\frac{1}{2\kappa}$.

\end{reptheorem}

\begin{proof}  Considering each of the $p$ regressions separately, we obtain consistency of the extended BIC with probability at least $1-n^{-\a}p^{-\b}$ via Theorem~\ref{thm:BIC_consistent}, as long as all the conditions~(B1)-(B5) hold. Using our assumptions for this current theorem, all these conditions hold by assumption, except for the eigenvalue bounds on $\EE\left[X_{1J}^{}X_{1J}^T\right]$ for all $|J|\leq 2q$, which we now derive  from properties of the logistic model combined with the conditions assumed to be true.

We need to find constants $a_1,a_2>0$ such that, for all $|J|\leq 2q$, $a_1\mathbf{I}_J\preceq \EE\left[X_{1J}^{}X_{1J}^T\right]\preceq a_2\mathbf{I}_J$. We now show that setting $a_1=\frac{1}{2q}\frac{e^{a_3(q+1)}}{\left(1+e^{a_3(q+1)}\right)}$ and $a_2=2q$ will satisfy this bound.

Fix any unit vector $u$ with support on $|J|\leq 2q$. We will show that $a_1\leq \EE\left[(X_{1J}^Tu)^2\right]\leq a_2$. Since $(X_{11},\dots,X_{1p})$ takes values in $\{0,1\}^p$, we have $\EE\left[(X_{1J}^Tu)^2\right]\leq \|u\|^2_1\leq 2q\|u\|^2_2=2q=a_2$. Next, we find a lower bound. Choose $j_0$ to maximize $u_{j_0}^2$; then $u_{j_0}^2\geq \frac{1}{2q}$. Let $J_0=J\backslash\{j_0\}$. We have
\begin{align*}
\EE\left[(X_{1J}^Tu)^2\right]&= \EE\left[\EE\left[(X_{1J}^Tu)^2 | X_{1J_0}\right]\right]\\
&= \EE\left[(X_{1J_0}^Tu_{j_0})^2+ 2 (X_{1J_0}^Tu_{J_0}) u_{j_0}\EE\left[X_{1j_0}|X_{1J_0}\right]+ u_{j_0}^2\EE\left[X_{1j_0}^2|X_{1J_0}\right]\right]\\
& = \EE\left[((X_{1J_0};\EE\left[X_{1j_0}|X_{1J_0}\right])^Tu_{j_0})^2+ u_{j_0}^2\left(\EE\left[X_{1j_0}^2|X_{1J_0}\right] - \EE\left[X_{1j_0}|X_{1J_0}\right]^2\right)\right]\\
& = \EE\left[((X_{1J_0};\EE\left[X_{1j_0}|X_{1J_0}\right])^Tu_{j_0})^2+ u_{j_0}^2\mathrm{Var}\left(X_{1j_0}|X_{1J_0}\right)\right]\\
& \geq  u_{j_0}^2\EE\left[\mathrm{Var}\left(X_{1j_0}|X_{1J_0}\right)\right]\\
& \geq  \frac{1}{2q}\EE\left[\mathrm{Var}\left(X_{1j_0}|X_{1J_0}\right)\right]\;.
\end{align*}

Now take any fixed value of $x_{[p]\backslash\{j_0\}}$. Using the logistic model,
$$\mathrm{Var}\left(X_{1j_0}|X_{1,[p]\backslash\{j_0\}}=x_{[p]\backslash\{j_0\}}\right) = \frac{\exp\left\{\zeta_{j_0}+\sum_{k\neq j_0} x_k \Theta^*_{j_0k}\right\}}{\left(1+\exp\left\{\zeta_{j_0}+\sum_{k\neq j_0} x_k \Theta^*_{j_0k}\right\}\right)^2}\;.$$
We also have
$\left|\zeta_{j_0}+\sum_{k:(j_0,k)\in G^*}x_k \Theta^*_{j_0k}\right|\leq |\zeta_{j_0}|+q \sup_{j,k}|\Theta^*_{jk}| \leq a_3(q+1)$,
and so
$$\mathrm{Var}\left(X_{1j_0}|X_{1,[p]\backslash\{j_0\}}=x_{[p]\backslash\{j_0\}}\right) \geq  \min_{|t|\leq a_3(q+1)}\frac{e^t}{(1+e^t)^2} = \frac{e^{a_3(q+1)}}{\left(1+e^{a_3(q+1)}\right)}\;.$$
Since this is true for any $x_{[p]\backslash\{j_0\}}$, we therefore have $\mathrm{Var}(X_{1j_0} | X_{1J_0})\geq \frac{e^{a_3(q+1)}}{\left(1+e^{a_3(q+1)}\right)}$ everywhere, and so
$$\EE\left[(X_{1\bul}^Tu)^2\right]\geq \frac{1}{2q}\EE\left[\mathrm{Var}\left(X_{1j_0}|X_{1J_0}\right)\right] \geq \frac{1}{2q}\frac{e^{a_3(q+1)}}{\left(1+e^{a_3(q+1)}\right)} = a_1\;.\qedhere$$
\end{proof}

\section{Proof of Lemma~\ref{lem:WHP}}\label{appendix:WHP}

\begin{replemma}{lem:WHP}
  Fix any $\a,\b>0$.  Assume (B1)-(B5) hold, and that either (A1) or
  (A2) holds.  For sufficiently large $n$, with probability at least
  $1-n^{-\a}p^{-\b}$ under (A1), or with probability at least
  $1-n^{-\a}p^{-\b}-4K^{K+1}n^{-\frac{K-2\kappa}{2}}$ under (A2), the
  following statements are all true.  The symbols $C_1$, $C_2$,
  $\l_1^*$, $\tau$, $R$, $\l_1$, $\l_2$, and $\l_3$ appearing in the
  statements represent constants that do not depend on $n$, $p$, or on
  the data, but generally are functions of other constants appearing
  in our assumptions.
\begin{itemize}
\item[(i)]The gradient of the likelihood is bounded at the true parameter vector $\phi^*$:
\begin{equation}\label{eq:score}\left\|\left(H_J(\phi^*)^{-\nicefrac{1}{2}}\right)s_J(\phi^*)\right\|_{2}<\sqrt{2\left(1+\e_n\right)\left|J\backslash J^*\right|\log\left(n^{\a}p^{1+\b}\right)}\text{ for all $J\supsetneq J^*$ with $|J|\leq 2q$}\;,\end{equation}
where $\e_n=C_1\sqrt{\frac{\log\left(n^{\a}p^{1+\b}\right)}{n}}+C_2\frac{1}{\log(n)}=\mathbf{o}(1)$.
\item[(ii)] Likelihood is upper-bounded by a quadratic function:
\begin{equation}\label{eq:phi_dist}\log \left(\frac{L_{[n]}(\phi^*+\psi_J)}{L_{[n]}(\phi^*)}\right)\leq -\frac{\l_1^*n}{2}\|\psi_J\|_2\left(\min\{1,\|\psi_J\|_2\}-\tau\sqrt{\frac{\log\left(n^{\a}p^{1+\b}\right)}{n}}\right)\text{ for all $|J|\leq 2q$, $\psi_J\in\R^J$}\;.\end{equation}
\item[(iii)]For all sparse models, the MLE lies inside a compact set:
\begin{equation}\label{eq:phat}\big\|\phatj\big\|_2\leq R\text{ for all $|J|\leq 2q$}\;.\end{equation}
\item[(iv)]The eigenvalues of the Hessian are bounded from above and below, and local changes in the Hessian are bounded from above, on the relevant compact set:
\begin{align}&\label{eq:Hessian}\text{For all $|J|\leq 2q$, $\|\phi_J\|_2\leq R+1$, } \lambda_1\mathbf{I}_J\preceq \tfrac{1}{n}H_J(\phi_J)\preceq \lambda_2 \mathbf{I}_J\;,\\
&\label{eq:Hessian_diff}\text{and for all $\|\phi_J\|_2,\|\phi'_J\|_2\leq R+1$, } \tfrac{1}{n}\left(H_J(\phi_J)-H_J(\phi'_J)\right)\preceq \|\phi_J-\phi'_J\|_2 \lambda_3\mathbf{I}_J\;.\end{align}
\end{itemize}\end{replemma}

We present the proofs of the various claims separately.

\subsection{Bound on the score at $\phi^*$: proving~(\ref{eq:score})}\label{appendix:score}

The following lemma is proved later, in Section~\ref{appendix:Hessian}.
\begin{lemma}\label{lem:Hessian}
Fix any radius $r>0$. There exist finite positive constants $c$, $\b_1=\b_1(r)$, $\b_2=\b_2(r)$, and $\b_3=\b_3(r)$ such that
\begin{align*}
&\b_1\mathbf{I}_J\preceq\tfrac{1}{n}H_J(\phi_J)\preceq \b_2\mathbf{I}_J\text{ for all $|J|\leq 2q$ and all $\|\phi_J\|_2\leq r$,}\\
&\text{and }\tfrac{1}{n}\left(H_J(\phi_J)-H_J(\phi'_J)\right)\preceq \|\phi_J-\phi'_J\|_2\cdot \b_3\mathbf{I}_J\text{ for all $|J|\leq 2q$ and all $\|\phi_J\|_2,\|\phi'_J\|_2\leq r$,}
\end{align*}
 with probability at least $1-p^{2q}e^{-cn}$ under (A1) or  with probability at least $1-2K^{K+1}pn^{-\nicefrac{K}{2}}-p^{2q}e^{-cn}$ under (A2).\end{lemma}

For large $n$,  since $\log(p)=\mathbf{o}(n)$ and $q$ is constant, we have $p^{2q}e^{-cn}<e^{-cn/2}<\frac{1}{3}n^{-\a}p^{-\b}$. Therefore,
by Lemma~\ref{lem:Hessian}, with probability at  least $1-\frac{1}{3}n^{-\a}p^{-\b}$ under (A1) or  with probability at least $1-\frac{1}{3}n^{-\a}p^{-\b} -2K^{K+1}pn^{-\nicefrac{K}{2}}$ under (A2),
\begin{align}
\label{eq:HessStar1}&\l_1^*\mathbf{I}_J\preceq\tfrac{1}{n}H_J(\phi_J)\preceq \l_2^*\mathbf{I}_J\text{ for all $|J|\leq 2q$ and all $\|\phi_J\|_2\leq a_3+1$,}\\
\label{eq:HessStar2}&\text{and }\tfrac{1}{n}\left(H_J(\phi_J)-H_J(\phi'_J)\right)\preceq \|\phi_J-\phi'_J\|_2\cdot \l_3^*\mathbf{I}_J\text{ for all $|J|\leq 2q$ and all $\|\phi_J\|_2,\|\phi'_J\|_2\leq a_3+1$,}
\end{align}
where $\l_k^*\coloneqq \b_k(a_3+1)$ for $k=1,2,3$. For the remainder of these proofs we assume that~(\ref{eq:HessStar1}) and~(\ref{eq:HessStar2}) are true.

We now bound the magnitude of the score. (We adapt the proof from \citet{Chen:2011}).  By Lemma 2 of \citet{Chen:2011}, there is a constant $U_0$ such that, for all $J$ with $|J|\leq 2q$, there exists a set of unit vectors $\UJ\subset \R^J$ with $|\UJ|\leq U_0$,  such that for all $v\in\R^J$, $\|v\|_2\leq \sqrt[4]{1+\e}\max_{u\in\UJ}u^Tv$.

Now fix any $J\supsetneq J^*$ with $|J|\leq 2q$, and any $u\in\UJ$. Below, we will show that
\begin{align*}
&\PP\left\{u^T\left(H_J(\phi^*)^{-\nicefrac{1}{2}}\right)s_J(\phi^*)\geq \sqrt{2\sqrt{1+\e}\left|J\backslash J^*\right|\log\left(n^{\a}p^{1+\b}\right)}\right\}\\
&\leq \exp\left\{-\sqrt{1+\e}\left|J\backslash J^*\right|\log\left(n^{\a}p^{1+\b}\right)\left(1-\sqrt{\frac{{\sqrt{1+\e}\cdot 2q\log\left(n^{\a}p^{1+\b}\right)}}{(\l_1^*)^{3}(\l_3^*)^{-2}n}}\right)\right\}\;.
\end{align*}
By the definition of $\UJ$, we then have
\begin{align*}
&\PP\left\{\left\|\left(H_J(\phi^*)^{-\nicefrac{1}{2}}\right)s_J(\phi^*)\right\|_2 \geq \sqrt{2\sqrt{1+\e}\left|J\backslash J^*\right|\log\left(n^{\a}p^{1+\b}\right)} \cdot \sqrt[4]{1+\e}\right\}\\
&\leq \sum_{u\in\UJ}\PP\left\{u^T\left(H_J(\phi^*)^{-\nicefrac{1}{2}}\right)s_J(\phi^*)\geq \sqrt{2\sqrt{1+\e}\left|J\backslash J^*\right|\log\left(n^{\a}p^{1+\b}\right)}\right\}\\
&\leq \sum_{u\in\UJ} \exp\left\{-\sqrt{1+\e}\left|J\backslash J^*\right|\log\left(n^{\a}p^{1+\b}\right)\left(1-\sqrt{\frac{{\sqrt{1+\e}\cdot 2q\log\left(n^{\a}p^{1+\b}\right)}}{(\l_1^*)^{3}(\l_3^*)^{-2}n}}\right)\right\}\\
&\leq U_0 \exp\left\{-\sqrt{1+\e}\left|J\backslash J^*\right|\log\left(n^{\a}p^{1+\b}\right)\left(1-\sqrt{\frac{{\sqrt{1+\e}\cdot 2q\log\left(n^{\a}p^{1+\b}\right)}}{(\l_1^*)^{3}(\l_3^*)^{-2}n}}\right)\right\}\;.
\end{align*}
  Therefore,
   \begin{align*}
  &\PP\left\{\exists J\subset [p],J\supsetneq J^*,|J|\leq 2q,\left\|\left(H_J(\phi^*)^{-\nicefrac{1}{2}}\right)s_J(\phi^*)\right\|_{2}\geq \sqrt{2(1+\e)\left|J\backslash J^*\right|\log\left(n^{\a}p^{1+\b}\right)}\right\}\\ 
  &\leq\sum_{N=1}^{2q-|J^*|}\sum_{J\subset [p],J\supsetneq J^*,\left|J\backslash J^*\right|=N}\PP\left\{\left\|\left(H_J(\phi^*)^{-\nicefrac{1}{2}}\right)s_J(\phi^*)\right\|_{2}\geq  \sqrt{2\sqrt{1+\e}\left|J\backslash J^*\right|\log\left(n^{\a}p^{1+\b}\right)} \cdot \sqrt[4]{1+\e}\right\}\\
  &\leq\sum_{N=1}^{2q-|J^*|}U_0\cdot{p\choose N}\cdot \exp\left\{-\sqrt{1+\e}\left|J\backslash J^*\right|\log\left(n^{\a}p^{1+\b}\right)\left(1-\sqrt{\frac{{\sqrt{1+\e}\cdot 2q\log\left(n^{\a}p^{1+\b}\right)}}{(\l_1^*)^{3}(\l_3^*)^{-2}n}}\right)\right\}\\
  &\leq\sum_{N=1}^{2q-|J^*|}\exp\left\{-N\sqrt{1+\e}\log\left(n^{\a}p^{1+\b}\right)\left(1-\sqrt{\frac{{\sqrt{1+\e}\cdot 2q\log\left(n^{\a}p^{1+\b}\right)}}{(\l_1^*)^{3}(\l_3^*)^{-2}n}}\right)+N\log(p)+\log(U_0)\right\}\\
  &\leq\sum_{N=1}^{\infty}\exp\left\{-N\sqrt{1+\e}\log\left(n^{\a}p^{\b}\right)\left(1-\sqrt{\frac{{\sqrt{1+\e}\cdot 2q\log\left(n^{\a}p^{1+\b}\right)}}{(\l_1^*)^{3}(\l_3^*)^{-2}n}}-\alpha^{-1}\log_n(U_0)\right)\right\}\;.
 \end{align*}
 We can simplify the expression above, as long as $\e$ is large enough to allow us to remove the vanishing terms inside the parentheses. In fact, for
 $$\e= 3\sqrt{\frac{{4q\log\left(n^{\a}p^{1+\b}\right)}}{(\l_1^*)^{3}(\l_3^*)^{-2}n}}+6\alpha^{-1}\log_n(U_0)\coloneqq C_1\sqrt{\frac{\log\left(n^{\a}p^{1+\b}\right)}{n}}+C_2\frac{1}{\log(n)}\;,$$
we get
   \begin{align*}
  &\PP\left\{\exists J\subset [p],J\supsetneq J^*,|J|\leq 2q,\left\|\left(H_J(\phi^*)^{-\nicefrac{1}{2}}\right)s_J(\phi^*)\right\|_{2}\geq \sqrt{2(1+\e)\left|J\backslash J^*\right|\log\left(n^{\a}p^{1+\b}\right)}\right\}\\ 
 &\leq\sum_{N=1}^{\infty}\exp\left\{-N\sqrt{1+\e}\log\left(n^{\a}p^{\b}\right)\left(1-\sqrt{\frac{{\sqrt{1+\e}\cdot 2q\log\left(n^{\a}p^{1+\b}\right)}}{(\l_1^*)^{3}(\l_3^*)^{-2}n}}-\alpha^{-1}\log_n(U_0)\right)\right\}\\
 &\leq\sum_{N=1}^{\infty}\exp\left\{-N\log\left(n^{\a}p^{\b}\right)-\log(6)\right\}
=\frac{n^{-\a}p^{-\b}}{6\left(1-n^{-\a}p^{-\b}\right)}\leq \frac{1}{3}n^{-\a}p^{-\b}\;,
 \end{align*}
which completes the proof, except that it remains to be shown that
\begin{align*}
&\PP\left\{u^T\left(H_J(\phi^*)^{-\nicefrac{1}{2}}\right)s_J(\phi^*)\geq \sqrt{2\sqrt{1+\e}\left|J\backslash J^*\right|\log\left(n^{\a}p^{1+\b}\right)}\right\}\\
&\leq \exp\left\{-\sqrt{1+\e}\left|J\backslash J^*\right|\log\left(n^{\a}p^{1+\b}\right)\left(1-\sqrt{\frac{{\sqrt{1+\e}\cdot 2q\log\left(n^{\a}p^{1+\b}\right)}}{(\l_1^*)^{3}(\l_3^*)^{-2}n}}\right)\right\}\;.
\end{align*}

The proof of this remaining inequality follows the techniques of \citet{Chen:2011}; we include it here for completeness, since we require a slightly more detailed analysis of the probabilities involved in order to obtain consistency results for the graphical models setting, as in Theorem~\ref{thm:Ising}.

Let $u\in\R^J$ be a unit vector. We now compute an upper bound on the quantity $u^T\left(H_J(\phi^*)^{-\nicefrac{1}{2}}\right)s_J(\phi^*)$ that holds with high probability. 
Since $s_J(\phi^*)=\sum_i X_{iJ}(Y_i-\mu_i)$, we have
$$u^T\left(H_J(\phi^*)^{-\nicefrac{1}{2}}\right)s_J(\phi^*)= \sum_i (Y_i-\mu_i)\cdot X_{iJ}^T\left(H_J(\phi^*)^{-\nicefrac{1}{2}}\right)u\;.$$

Next, for convenience we write $A\coloneqq \sqrt{2\sqrt{1+\e}\left|J\backslash J^*\right|\log\left(n^{\a}p^{1+\b}\right)}$ and $\psi_J=A\cdot \left(H_J(\phi^*)^{-\nicefrac{1}{2}}\right)u$.  Since $\mathrm{Var}\left(s_J(\phi^*)\right)=H_J(\phi^*)=\sum_i X_{iJ}X_{iJ}^T\bb''(X_{i\bul}^T\phi^*)$, we have 
 \begin{align*}
 \sum_i \left(X_{iJ}^T\psi_J\right)^2\cdot \bb''(X_{i\bul}^T\phi^*)
&= A^2\sum_i \left(X_{iJ}^T\left(H_J(\phi^*)^{-\nicefrac{1}{2}}\right)u\right)^2\cdot \bb''(X_{i\bul}^T\phi^*) \\
 &=A^2 u^T\left(H_J(\phi^*)^{-\nicefrac{1}{2}}\right)^T\left(\sum_i X_{iJ}X_{iJ}^T\bb''(X_{i\bul}^T\phi^*)\right)\left(H_J(\phi^*)^{-\nicefrac{1}{2}}\right)u\\
& =A^2\cdot u^T\left(H_J(\phi^*)^{-\nicefrac{1}{2}}\right)^T H_J(\phi^*) \left(H_J(\phi^*)^{-\nicefrac{1}{2}}\right) u
 =A^2\cdot u^Tu = A^2\;.\end{align*}
And,
$$\|\psi_J\|^2_2=A^2\left\|\left(H_J(\phi^*)^{-\nicefrac{1}{2}}\right)u\right\|^2_2
\leq A^2\cdot \left\|H_J(\phi^*)^{-1}\right\|_{\mathrm{sp}}\cdot\|u\|^2_2\leq A^2(\l_1^*n)^{-1}\;.$$
We then have
\begin{align*}
&\PP\left\{u^T\left(H_J(\phi^*)^{-\nicefrac{1}{2}}\right)s_J(\phi^*)\geq A\right\}\\
&=\EE\left[\one{A\sum_i (Y_i-\mu_i)\cdot X_{iJ}^T\left(H_J(\phi^*)^{-\nicefrac{1}{2}}\right)u\geq A^2}\right]\\
&=\EE\left[\one{\sum_i (Y_i-\mu_i)\cdot X_{iJ}^T\psi_J\geq A^2}\right]\\
&\leq \EE\left[\exp\left\{\sum_i (Y_i-\mu_i)\cdot X_{iJ}^T\psi_J-A^2\right\}\right]\\
\displaybreak&= \exp\left\{-A^2-\sum_i \mu_i\cdot X_{iJ}^T\psi_J\right\}\cdot \EE\left[\exp\left\{\sum_i Y_i\cdot X_{iJ}^T\psi_J\right\}\right]\\
&= \exp\left\{-A^2-\sum_i \mu_i\cdot X_{iJ}^T\psi_J\right\}\cdot\prod_i  \EE\left[\exp\left\{Y_i\cdot X_{iJ}^T\psi_J\right\}\right]\\
&= \exp\left\{-A^2-\sum_i \mu_i\cdot X_{iJ}^T\psi_J\right\}\cdot\prod_i
\exp\left\{ \left[\bb\left(X_{i\bul}^T\left(\phi^*+\psi_J\right)\right)-\bb\left(X_{i\bul}^T\phi^*\right)\right]\right\}\\
&= \exp\left\{-A^2-\sum_i \mu_i\cdot X_{iJ}^T\psi_J\right\}\cdot
\exp\left\{\sum_i \left[\bb\left(X_{i\bul}^T\left(\phi^*+\psi_J\right)\right)-\bb\left(X_{i\bul}^T\phi^*\right)\right]\right\}\;,
\end{align*}
where the next-to-last step comes from the properties of exponential families.

By the Taylor series approximation, for some $t\in[0,1]$, 
\begin{align*}
&\sum_i \left[\bb\left(X_{i\bul}^T\left(\phi^*+\psi_J\right)\right)-\bb\left(X_{i\bul}^T\phi^*\right)\right]\\
&=\sum_i \bb'\left(X_{i\bul}^T\phi^*\right)\cdot X_{i\bul}^T\psi_J+ \frac{1}{2}\bb''\left(X_{i\bul}^T\phi^*\right)\cdot (X_{i\bul}^T\psi_J)^2
+\frac{1}{2}(X_{i\bul}^T\phi^*)^2\left(\bb''\left(X_{i\bul}^T(\phi^*+t\psi_J)\right)-\bb''\left(X_{i\bul}^T\phi^*\right)\right)\\
&= \left(\sum_i \mu_i\cdot X_{i\bul}^T\psi_J + \frac{1}{2}\bb''\left(X_{i\bul}^T\phi^*\right)\cdot (X_{i\bul}^T\psi_J)^2\right)
+\frac{1}{2}\psi_J^T\left(\sum_i X_{iJ}X_{iJ}^T\left(\bb''\left(X_{i\bul}^T(\phi^*+t\psi_J)\right)-\bb''\left(X_{i\bul}^T\phi^*\right)\right)\right)\psi_J\\
&= \left(\sum_i \mu_i\cdot X_{i\bul}^T\psi_J\right)+\frac{A^2}{2}
+\frac{1}{2}\psi_J^T\left(H_J(\phi^*+t\psi_J)-H_J(\phi^*)\right)\psi_J
\leq \left(\sum_i \mu_i\cdot X_{i\bul}^T\psi_J \right)+\frac{A^2}{2}
+\frac{1}{2}\|\psi_J\|^3_2\cdot n\l_3^*\\
&\leq \left(\sum_i \mu_i\cdot X_{i\bul}^T\psi_J \right)+\frac{A^2}{2}
+\frac{A^3\l_3^*}{2(\l_1^*)^{1.5}n^{0.5}}\;.
\end{align*}
Continuing from above, we obtain the desired inequality as follows:
\begin{align*}
&\PP\left\{u^T\left(H_J(\phi^*)^{-\nicefrac{1}{2}}\right)s_J(\phi^*)\geq A\right\}\\
&\leq \exp\left\{-A^2-\sum_i \mu_i\cdot X_{iJ}^T\psi_J\right\}\cdot
\exp\left\{\sum_i \left[\bb\left(X_{i\bul}^T\left(\phi^*+\psi_J\right)\right)-\bb\left(X_{i\bul}^T\phi^*\right)\right]\right\}\\
&\leq \exp\left\{-\frac{A^2}{2}+\frac{A^3\l_3^*}{2(\l_1^*)^{1.5}n^{0.5}}\right\}\\
&= \exp\left\{-\sqrt{1+\e}\left|J\backslash J^*\right|\log\left(n^{\a}p^{1+\b}\right)\left(1-\sqrt{\frac{{\sqrt{1+\e}\cdot 2q\log\left(n^{\a}p^{1+\b}\right)}}{(\l_1^*)^{3}(\l_3^*)^{-2}n}}\right)\right\}\;.
\end{align*}

\subsection{Accuracy of MLE for true sparse models: proving~(\ref{eq:phi_dist})}

Assume that~(\ref{eq:HessStar1}), (\ref{eq:HessStar2}) and~(\ref{eq:score}) hold. Fix $J$ with $J\supset J^*$, $|J|\leq 2q$, and fix any $\psi_J$ with $\|\psi_J\|_2\leq 1$. Then
\begin{align*}&\log L_{[n]}(\phi^*+\psi_J)-\log L_{[n]}(\phi^*)=\psi_J^T s_{J}(\phi^*)-\frac{1}{2}\psi_J^T H_J(\phi^*+t\psi_J)\psi_J\\
&\leq \|\psi_J\|_2\cdot\|s_J(\phi^*)\|_2-\|\psi_J\|_2^2\cdot \frac{\l_1^*n}{2}\\
&\leq \|\psi_J\|_2\cdot\left\|\left(H_J(\phi^*)^{-\nicefrac{1}{2}}\right)s_J(\phi^*)\right\|_2\cdot \left\|H_J(\phi^*)\right\|^{\nicefrac{1}{2}}_{\mathrm{sp}}-\|\psi_J\|_2^2\cdot\frac{\l_1^*n}{2}\\
&\leq \|\psi_J\|_2\cdot\sqrt{2(1+\e)\left|J\backslash J^*\right|\log\left(n^{\a}p^{1+\b}\right)}\cdot \sqrt{\l_2^* n}-\|\psi_J\|_2^2\cdot\frac{\l_1^* n}{2}\\
&\leq -\frac{\l_1^*n}{2}\|\psi_J\|_2\left(\|\psi_J\|_2-\sqrt{\frac{\log\left(n^{\a}p^{1+\b}\right)}{n}}\cdot \sqrt{16(1+\e)q\l_2^*(\l_1^*)^{-2}}\right)\\
&= -\frac{\l_1^*n}{2}\|\psi_J\|_2\left(\|\psi_J\|_2-\tau\sqrt{\frac{\log\left(n^{\a}p^{1+\b}\right)}{n}}\right)\;,
\end{align*}
where $\tau=\sqrt{16(1+\e)q\l_2^*(\l_1^*)^{-2}}$. By convexity of the log-likelihood, this means that for all $\psi_J\in\R^J$,
$$\log L_{[n]}(\phi^*+\psi_J)-\log L_{[n]}(\phi^*)\leq -\frac{\l_1^*n}{2}\|\psi_J\|_2\left(\min\{1,\|\psi\|_2\}-\tau\sqrt{\frac{\log\left(n^{\a}p^{1+\b}\right)}{n}}\right)\;,$$
which proves (\ref{eq:phi_dist}). In particular, since $\log L_{[n]}(\phatj)\geq \log L_{[n]}(\phi^*)$, applying the convexity of log-likelihood, we must have (for sufficiently large $n$)
$$\|\phatj-\phi^*\|_2\leq \sqrt{\frac{\log\left(n^{\a}p^{1+\b}\right)}{n}}\cdot \tau\;.$$

\subsection{Compact set containing all sparse MLEs: proving~(\ref{eq:phat}),~(\ref{eq:Hessian}), and~(\ref{eq:Hessian_diff})}

Assume that~(\ref{eq:HessStar1}), (\ref{eq:HessStar2}), (\ref{eq:score}), and~(\ref{eq:phi_dist}) hold. Let
$$R\coloneqq 1+a_3+4(\l_1^*)^{-1}\left(\sqrt{2(1+\e)(1+\a+\b)a_3^2q \l_2^*}+\frac{1}{2}a_3^2\l_3^*\right)\;.$$
We now show that $\|\phatj\|_2\leq R$ for all $|J|\leq 2q$. We will use the fact that, since the zero coefficient vector $\mathbf{0}\in\R^p$ is contained in every model $J$, the coefficient vector $\phatj$ must yield higher likelihood than the vector $\mathbf{0}$. 

 First, we compute a lower bound for $L_{[n]}(\mathbf{0})$:
\begin{align*}\log L_{[n]}(\mathbf{0})-\log L_{[n]}(\phi^*)&=(-\phi^*)^T s_{J^*}(\phi^*)-\frac{1}{2}(-\phi^*)^T H_{J^*}(t\phi^*)(-\phi^*)\\
&=-\left(H_{J^*}(\phi^*)^{\nicefrac{1}{2}}\phi^*\right)^T\left(H_{J^*}(\phi^*)^{-\nicefrac{1}{2}}s_{J^*}(\phi^*)\right)-\frac{1}{2}\phi^*{}^T H_{J^*}(t\phi^*)\phi^*\\
&\geq -\sqrt{\phi^*{}^T H_{J^*}(\phi^*)\phi^*}\left\|H_{J^*}(\phi^*)^{-\nicefrac{1}{2}}s_{J^*}(\phi^*)\right\|_2 -\frac{1}{2}\phi^*{}^T H_{J^*}(t\phi^*)\phi^*\\
&\geq -\sqrt{a_3^2\cdot n\l_2^*}\cdot\sqrt{2(1+\e)\left|J\backslash J^*\right|\log\left(n^{\a}p^{1+\b}\right)}-\frac{1}{2}a_3^2\cdot n\l_2^*\\
&=-n\left(\sqrt{2(1+\e)(1+\a+\b)a_3^2 q \l_2^*}\sqrt{\frac{\log\left(n^{\a}p^{1+\b}\right)}{n}}+\frac{1}{2}a_3^2\l_2^*\right)\\
&\geq -n\left(\sqrt{2(1+\e)(1+\a+\b)a_3^2 q \l_2^*}+\frac{1}{2}a_3^2 \l_2^*\right)\;,
\end{align*}
for sufficiently large $n$, since $\log\left(n^{\a}p^{1+\b}\right)=\mathbf{o}(n)$.

Next, we consider $L_{[n]}(\phatj)$, and find that since $\log L_{[n]}(\phatj)\geq \log L_{[n]}(\mathbf{0})$ by definition, this results in a bound on $\|\phatj\|_2$. Fix any $J$ with $|J|\leq q$. If $\|\phatj-\phi^*\|_2\leq 1$, then $\|\phatj\|_2\leq 1+\|\phi^*\|_2\leq 1+a_3\leq R$. Now consider the case that $\|\phatj-\phi^*\|_2\geq 1$.  Applying (\ref{eq:phi_dist}) to the model $J'\coloneqq J\cup J^*$ with $\psi_{J'}\coloneqq \phatj-\phi^*$, we obtain
\begin{align*}
\log L_{[n]}(\phi^*+(\phatj-\phi^*))-\log L_{[n]}(\phi^*)&\leq-\frac{\l_1^*n}{2}\|\phatj-\phi^*\|_2\left(\min\left\{1,\|\phatj-\phi^*\|_2\right\}-\tau\sqrt{\frac{\log\left(n^{\a}p^{1+\b}\right)}{n}}\right)\\
&=-\frac{\l_1^*n}{2}\|\phatj-\phi^*\|_2\left(1-\tau\sqrt{\frac{\log\left(n^{\a}p^{1+\b}\right)}{n}}\right)\\
&\leq-\frac{\l_1^*n}{4}\|\phatj-\phi^*\|_2\;,
\end{align*}
for sufficiently large $n$, since $\log(p)=\mathbf{o}(n)$. Combining these results, we obtain 
\begin{align*}
-n\left(\sqrt{2(1+\e)(1+\a+\b)a_3^2 q \l_2^*}+\frac{1}{2}a_3^2\l_2^*\right)&\leq \log L_{[n]}(\mathbf{0})-\log L_{[n]}(\phi^*)\\
&\leq \log L_{[n]}(\phatj)-\log L_{[n]}(\phi^*)\\
&\leq -\frac{\l_1^*n}{4} \|\phatj-\phi^*\|_2\;.\end{align*}
Therefore,
$$\|\phatj\|_2\leq \|\phi^*\|_2+\|\phatj-\phi^*\|_2\leq a_3+4(\l_1^*)^{-1}\left(\sqrt{2(1+\e)(1+\a+\b)a_3^2 q \l_2^*}+\frac{1}{2}a_3^2\l_2^*\right)\leq R\;.$$

Finally, define
$\lambda_k=\beta_k(R+1)$ for $k=1,2,3$. As in Section~\ref{appendix:score}, we apply  Lemma~\ref{lem:Hessian} and see that with probability at  least $1-\frac{1}{3}n^{-\a}p^{-\b} -2K^{K+1}pn^{-\nicefrac{K}{2}}$ under (A1) or  with probability at least $1-\frac{1}{3}n^{-\a}p^{-\b} -2K^{K+1}pn^{-\nicefrac{K}{2}}$ under (A2),
\begin{align*}
&\l_1\mathbf{I}_J\preceq\tfrac{1}{n}H_J(\phi_J)\preceq \l_2\mathbf{I}_J\text{ for all $|J|\leq 2q$ and all $\|\phi_J\|_2\leq R+1$,}\\
&\text{and }\tfrac{1}{n}\left(H_J(\phi_J)-H_J(\phi'_J)\right)\preceq \|\phi_J-\phi'_J\|_2\cdot \l_3\mathbf{I}_J\text{ for all $|J|\leq 2q$ and all $\|\phi_J\|_2,\|\phi'_J\|_2\leq R+1$.}
\end{align*}

\section{Proof of Lemma~\ref{lem:Hessian}}\label{appendix:Hessian}

We now prove the bounds on the Hessian.
\begin{replemma}{lem:Hessian}
Fix any radius $r>0$. There exist finite positive constants $c$, $\b_1=\b_1(r)$, $\b_2=\b_2(r)$, and $\b_3=\b_3(r)$ such that
\begin{align*}
&\b_1\mathbf{I}_J\preceq\tfrac{1}{n}H_J(\phi_J)\preceq \b_2\mathbf{I}_J\text{ for all $|J|\leq 2q$ and all $\|\phi_J\|_2\leq r$,}\\
&\text{and }\tfrac{1}{n}\left(H_J(\phi_J)-H_J(\phi'_J)\right)\preceq \|\phi_J-\phi'_J\|_2\cdot \b_3\mathbf{I}_J\text{ for all $|J|\leq 2q$ and all $\|\phi_J\|_2,\|\phi'_J\|_2\leq r$,}
\end{align*}
 with probability at least $1-p^{2q}e^{-cn}$ under (A1) or  with probability at least $1-2K^{K+1}pn^{-\nicefrac{K}{2}}-p^{2q}e^{-cn}$ under (A2).
\end{replemma}
\begin{proof} Under (A1), $\EE\left[\left|X_{1j}\right|^4\right]\leq \AA^4$, while under (A2), $\EE\left[\left|X_{1j}\right|^4\right]\leq \AA_K^{\nicefrac{2}{(3K)}}$. Define $m$ to equal $\AA^4$ or $\AA_K^{\nicefrac{2}{(3K)}}$, as appropriate. By Lemma~\ref{lem:pos_def_H} below, with probability at least $1-{p\choose 2q}e^{- \left(150\cdot \left\lceil80q^2ma_1^{-2}\right\rceil\right)^{-1}n}$, for all $J$ with $|J|=2q$ and for all $\phi$ with $\|\phi\|_2\leq r$,
$$H_J(\phi)\succeq n\mathbf{I}_J\cdot \frac{a_1}{4} \inf\left\{\bb''(\theta):|\theta|\leq 20q^2r\sqrt{m}\left\lceil 80q^2ma_1^{-2}\right\rceil\right\}\;.$$

Now we show an upper bound and bound the difference. By Lemma~\ref{lem:upper_bound_H} below, with probability one under (A1) or with probability at least $1-2K^{K+1}pn^{-\nicefrac{K}{2}}$ under (A2), for all $J$ with $|J|\leq 2q$ and all $\phi_J,\phi'_J$ with $\|\phi_J\|_2,\|\phi'_J\|_2\leq r$,
$$H_J(\phi_J)-H_J(\phi'_J)\preceq \|\phi_J-\phi'_J\|_2\cdot C_1(r)\cdot n\mathbf{I}_J\;.$$
In particular, this implies that
$$H_J(\phi_J)=H_J(\phi_J)-H_J(\mathbf{0}_J)\preceq \|\phi_J-\mathbf{0}_J\|_2\cdot C_1(r)\cdot n\mathbf{I}_J\preceq  rC_1(r)\cdot n\mathbf{I}_J\;.$$
Let $\b_1(r)\coloneqq \frac{a_1}{4} \inf\left\{\bb''(\theta):|\theta|\leq 20q^2r\sqrt{m}\left\lceil 80q^2ma_1^{-2}\right\rceil\right\}$, and $\b_2(r),\b_3(r)\coloneqq C_1(r)$. This proves the claim.
\end{proof}

\subsection{Bounding the change in the Hessian when $x$'s are subgaussian}

\begin{lemma}\label{lem:upper_bound_H}
For any radius $r>0$, there exists finite  $C_1=C_1(r)$ such that for any sample under (A1), or with probability at least $1-2K^{K+1}pn^{-\nicefrac{K}{2}}$ under (A2), for all $J$ with $|J|\leq 2q$, for all $\phi_J,\phi'_J$ with $\|\phi_J\|_2,\|\phi'_J\|\leq r$,
$$\tfrac{1}{n}\left(H_J(\phi_J)-H_J(\phi'_J)\right)\preceq \|\phi_J-\phi'_J\|_2 C_1\mathbf{I}_J\;.$$
\end{lemma}
\begin{proof}

For some convex combination $\phi''_J=t\phi_J+(1-t)\phi'_J$,
\begin{align*}
\left\|H_J(\phi_J)-H_J(\phi'_J)\right\|_{\mathrm{sp}}
&\leq\left\|H_J(\phi_J)-H_J(\phi'_J)\right\|_F\\
&=\left\|\sum_i X_{iJ}X_{iJ}^T\left(\bb''(X_{i\bul}^T\phi_J)-\bb''(X_{i\bul}^T\phi'_J)\right)\right\|_F\\
&=\left\|\sum_i X_{iJ}X_{iJ}^T\cdot \bb'''(X_{i\bul}^T\phi''_J)\cdot (X_{i\bul}^T(\phi_J-\phi'_J))\right\|_F\\
&\leq\sum_i \left\|X_{iJ}X_{iJ}^T\right\|_F\cdot\left|\bb'''(X_{i\bul}^T\phi''_J)\right|\cdot \left|(X_{i\bul}^T(\phi_J-\phi'_J))\right|\\
&\leq\left\|\phi_J-\phi'_J\right\|_2\cdot\sum_i \left\|X_{iJ}\right\|^3_2\cdot \left|\bb'''(X_{i\bul}^T\phi''_J)\right|\;.
\end{align*}

Under assumption (A1), since $|X_{ij}|\leq \AA$ for all $i,j$,
$$\left\|H_J(\phi_J)-H_J(\phi'_J)\right\|_{\mathrm{sp}}\leq \left\|\phi_J-\phi'_J\right\|_2n\cdot\left( (2q)^{1.5}\AA^3\cdot\sup_{|\theta|\leq \AA\sqrt{2q}\cdot r}\left|\bb'''(\theta)\right|^2\right)\coloneqq  \left\|\phi_J-\phi'_J\right\|_2n\cdot C_1\;.$$

Now we turn to the setting of assumption (A2). 
By inequality (86) of \citet{Ravikumar:2011}, if $W_1,\dots,W_n$ are i.i.d.\ copies of a random variable $W$ with $\EE[|W|^K]\leq M$, then
\begin{align*}
\EE\Big[\Big|\sum_i W_i-\EE[W]\Big|^{K}\Big]&\leq n^{\nicefrac{K}{2}}\left(\nicefrac{K}{2}\right)^{K+1}\EE\left[\left|W-\EE[W]\right|^{K}\right]\\
&\leq  n^{\nicefrac{K}{2}}\left(\nicefrac{K}{2}\right)^{K+1}\cdot 2^{K}\left(\EE[|W|^{K}]+\left|\EE[W]\right|^{K}\right)\\&\leq n^{\nicefrac{K}{2}}K^{K+1} M\;,\end{align*}
and therefore,
\begin{align*}
\PP \Big\{\frac{1}{n}\sum_i W_i> 2M^{\nicefrac{1}{K}} \Big\}
&\leq \PP \Big\{\Big |\sum_i W_i-\EE[W] \Big |>nM^{\nicefrac{1}{K}} \Big\}\\
&= \PP \Big\{\Big |\sum_i W_i-\EE[W] \Big |^K>n^KM \Big\}\\
&\leq \frac{\EE\left[\left|\sum_i W_i-\EE[W]\right|^K\right]}{n^KM}\\
&\leq \frac{n^{\nicefrac{K}{2}}K^{K+1} M}{n^KM}\\&=K^{K+1}n^{-\nicefrac{K}{2}}\;.\end{align*}
We apply this result $2p$ times, to obtain that  with probability at least $1-2p\cdot K^{K+1}n^{-\nicefrac{K}{2}}$, for all $j\in[p]$,
$$\sum_i |X_{ij}|^6\leq 2n\AA_K^{\nicefrac{1}{K}}\;, \text{ and }\sum_i  \sup_{|\theta|\leq r\sqrt{2q}|X_{ij}|}\left|\bb'''(\theta)\right|\leq 2n\BB_K(r\sqrt{2q})^{\nicefrac{1}{K}}\;.$$
Now assume that both of these bounds hold for every $j$.
Then, for each $|J|\leq 2q$,
$$\sum_i \|X_{iJ}\|^6_2\leq (2q)^3\max_{j\in J}\sum_i \left|X_{ij}\right|^6\leq (2q)^3\cdot 2n \AA_K^{\nicefrac{1}{K}}\;.$$
Finally, for each $|J|\leq 2q$, observe that for each $i$, $\left|X_{i\bul}^T\phi''_J\right|\leq r\|X_{iJ}\|_2\leq r\sqrt{2q}\max_{j\in J}|X_{ij}|$, and so
$$\sum_i \left|\bb'''(X_{i\bul}^T\phi''_J)\right|^2\leq \max_{j\in J} \sum_i  \sup_{|\theta|\leq r\sqrt{2q}|X_{ij}|}\left|\bb'''(\theta)\right|\leq 2n\BB_K\left(r\sqrt{2q}\right)^{\nicefrac{1}{K}}\;.$$
Therefore,\begin{align*}
\left\|H_J(\phi_J)-H_J(\phi'_J)\right\|_{\mathrm{sp}}&\leq\left\|\phi_J-\phi'_J\right\|_2\cdot\sum_i \left\|X_{iJ}\right\|^3_2\cdot \left|\bb'''(X_{i\bul}^T\phi''_J)\right|\\
&\leq\left\|\phi_J-\phi'_J\right\|_2\cdot\frac{1}{2}\sum_i \left(\left\|X_{iJ}\right\|^6_2+ \left|\bb'''(X_{i\bul}^T\phi''_J)\right|^2\right)\\
&\leq \left\|\phi_J-\phi'_J\right\|_2n\cdot \left((2q)^3 \AA_K^{\nicefrac{1}{K}}+\BB_K\left(r\sqrt{2q}\right)^{\nicefrac{1}{K}}\right)\\&\coloneqq \left\|\phi_J-\phi'_J\right\|_2n\cdot C_1\;.\qedhere\end{align*}
\end{proof}

\subsection{Positive definite Hessian}
We now show that, under mild assumptions, the Hessian of the negative log-likelihood will be positive definite with its smallest eigenvalue bounded away from zero.

\begin{replemma}{lem:pos_def_H}
Fix $J$ with $|J|=2q$, and radius $R>0$. Assume
 $\lambda_{\min}\left(\EE\left[X_{1J}^{}X_{1J}^T\right]\right)\geq a_1>0$ and
 $\sup_{j\in J} \EE\left[\left|X_{1j}\right|^4\right]\leq m$.
If $n$ is sufficiently large, then with probability at least $1-e^{-\left(150\cdot \left\lceil80q^2ma_1^{-2}\right\rceil\right)^{-1}n}$, for all $\phi=\phi_J$ with $\|\phi\|_2\leq r$, 
$$H_J(\phi)\succeq n\mathbf{I}_J\cdot \frac{a_1}{4} \inf\left\{\bb''(\theta):|\theta|\leq 20q^2r\sqrt{m}\left\lceil 80q^2ma_1^{-2}\right\rceil\right\}\;.$$
\end{replemma}

We first give a brief intuition for the proof. We have $H_J(\phi)=\sum_i X_{iJ}X_{iJ}^T\cdot\bb''(X_{i\bul}^T\phi)$.
Due to the moment condition on the covariates, we know that $\sum_i X_{iJ}X_{iJ}^T$ will be approximately equal to $n\EE\left[X_{1J}^{}X_{1J}^T\right]\succeq n\cdot a_1\mathbf{I}_J$. However, this is not sufficient, because for some $i$, we might have very small values of $\bb''(X_{i\bul}^T\phi)$. Instead, we consider only those $i$ for which $\bb''(X_{i\bul}^T\phi)$ satisfies some lower bound. By considering the sum of $X_{iJ}X_{iJ}^T$ over this subset of the $i$'s, we will obtain the desired result.

\begin{proof}
From the assumptions, for all $j,k\in J$,
$$\mathrm{Var}\left(X_{1j}X_{1k}\right)\leq \EE\left[\left|X_{1j}\right|^2\left|X_{1k}\right|^2\right]\leq \frac{1}{2}\EE\left[\left|X_{1j}\right|^4+\left|X_{1k}\right|^4\right]\leq m\;.$$
 Let $N=\left\lceil 80q^2ma_1^{-2}\right\rceil$, and let $n'=\left\lfloor \frac{n}{2N}\right\rfloor$. Then
$$N\geq \frac{20\sum_{j,k\in J}\mathrm{Var}\left(X_{1j}X_{1k}\right)}{\lambda^2_{\min}\left(\EE\left[X_{1J}^{}X_{1J}^T\right]\right)}\;.$$

For each $i_0=N,2N,3N,\dots,(2n')N$, define matrix $\mathbf{M}^{(i_0)}\in\R^{J\times J}$ as
$$\mathbf{M}^{(i_0)}_{jk}=\frac{1}{N}\sum_{i=i_0-(N-1)}^{i_0}X_{ij}X_{ik}-\EE\left[X_{1j}X_{1k}\right]\;,$$
and define events
\begin{align*}
&E^{(i_0)}=\left\{\|\mathbf{M}^{(i_0)}\|_{\mathrm{sp}}\leq \frac{1}{2}\lambda_{\min}\left(\EE\left[X_{1J}^{}X_{1J}^T\right]\right)\right\}\;,\\
&F^{(i_0)}=\left\{X_{ij}^2\leq 10q\sqrt{m}N \text{ for all } j\in J, i=i_0-(N-1),\dots,i_0\right\}\;.
\end{align*}
Define also positive constant
$$b_0=\inf_{|\theta|\leq 20q^2r\sqrt{m}N}\bb''(\theta)\;.$$

Below we show that, for the fixed choice of $J$, with probability at least $1-e^{-\left(150\cdot \left\lceil80q^2ma_1^{-2}\right\rceil\right)^{-1}n}$,
$$\#\left\{i_0\in \left\{N,2N,3N,\dots,(2n')N\right\} \ : \ E^{(i_0)}\cap F^{(i_0)}\right\}\geq \tfrac{n}{2N}\;.$$

Now suppose that this is true. Take any $i_0$ such that $E^{(i_0)}$ and $F^{(i_0)}$ both occur. Then, by definition of $E^{(i_0)}$,
\begin{align*}
&\frac{1}{N}\sum_{i=i_0-(N-1)}^{i_0}X_{ij}X_{ik}
=\mathbf{M}^{(i_0)}+\EE\left[X_{1J}^{}X_{1J}^T\right]=\mathbf{M}^{(i_0)}+\EE\left[X_{1J}\right]\EE\left[X_{1J}\right]^T+Cov(X_{1J})\\
&\succeq \mathbf{M}^{(i_0)}+Cov(X_{1J})\succeq \left(\lambda_{\min}(Cov(X_{1J}))-\|\mathbf{M}^{(i_0)}\|_{\mathrm{sp}}\right)\mathbf{I}_J
\succeq \frac{1}{2}\lambda_{\min}(Cov(X_{1J}))\mathbf{I}_J=\frac{a_1}{2}\mathbf{I}_J\;.
\end{align*}
And, by definition of $F^{(i_0)}$, for all $\phi$ with $\mathit{Support}(\phi)=J$ and $\|\phi\|_2\leq r$,
$$\bb''(X_{i\bul}^T\phi)\geq b_0 \text{ for all } i=i_0-(N-1),\dots,i_0\;.$$
Therefore, for all $\phi$ with $\mathit{Support}(\phi)=J$ and $\|\phi\|_2\leq r$,
\begin{align*}
H_J(\phi)=\sum_{i=1}^nX_{iJ}X_{iJ}^T\bb''(X_{i\bul}^T\phi)
&=\sum_{i_0=N,\dots,(2n')N}\sum_{i=i_0-(N-1)}^{i_0}X_{iJ}X_{iJ}^T\bb''(X_{i\bul}^T\phi)\\
&\succeq\sum_{i_0:E^{(i_0)},F^{(i_0)}}\sum_{i=i_0-(N-1)}^{i_0}X_{iJ}X_{iJ}^T\bb''(X_{i\bul}^T\phi)\\
&\succeq b_0\sum_{i_0:E^{(i_0)},F^{(i_0)}}\sum_{i=i_0-(N-1)}^{i_0}X_{iJ}X_{iJ}^T\\
&\succeq\sum_{i_0:E^{(i_0)},F^{(i_0)}} \frac{b_0a_1N}{2}\mathbf{I}_J\\
&\succeq \frac{n}{2N}\cdot \frac{b_0a_1N}{2}\mathbf{I}_J= \frac{b_0a_1}{4}\cdot n\mathbf{I}_J\;.
\end{align*}
\medskip

It remains to be shown that
$$\#\left\{i_0\in \left\{N,2N,3N,\dots,(2n')N\right\} \ : \ E^{(i_0)}\cap F^{(i_0)}\right\}\geq \tfrac{n}{2N}\;.$$
We will do this by showing that (for each $i_0$) $\PP\left\{E^{(i_0)}\right\}\geq 0.8$ and $\PP\left\{E^{(i_0)}\right\}\geq 0.8$.

Fix any  $i_0\in\{1,\dots,(2n')\}$. First, we treat the event $E^{(i_0)}$. By the definition of $N$, we have
\begin{align*}
\EE\left[\sum_{j,k\in J}\left(\frac{1}{N}\sum_{i=i_0-(N-1)}^{i_0}X_{ij}X_{ik}-\EE\left[X_{1j}X_{1k}\right]\right)^2\right]
&=\sum_{j,k\in J}\mathrm{Var}\left(\frac{1}{N}\sum_{i=i_0-(N-1)}^{i_0}X_{ij}X_{ik}\right)\\
&=\frac{1}{N}\sum_{j,k\in J}\mathrm{Var}\left(X_{1j}X_{1k}\right)\\
&\leq \frac{1}{20}\lambda_{\min}^2\left(\EE\left[X_{1J}^{}X_{1J}^T\right]\right)\;.
\end{align*}
Next, we define matrix $\mathbf{M}^{(i_0)}\in\R^{J\times J}$ as
$\mathbf{M}^{(i_0)}_{jk}=\frac{1}{N}\sum_{i=i_0-(N-1)}^{i_0}X_{ij}X_{ik}-\EE\left[X_{1j}X_{1k}\right]$.
We have, by Markov's inequality, since $\|\mathbf{M}^{(i_0)}\|_{\mathrm{sp}}\leq \|\mathbf{M}^{(i_0)}\|_F$,
\begin{align*}
 &\PP\left\{\|\mathbf{M}^{(i_0)}\|_{\mathrm{sp}}> \frac{1}{2}\lambda_{\min}\left(\EE\left[X_{1J}^{}X_{1J}^T\right]\right)\right\}\\
 &\leq \PP\left\{\|\mathbf{M}^{(i_0)}\|^2_F> \frac{1}{4}\lambda^2_{\min}\left(\EE\left[X_{1J}^{}X_{1J}^T\right]\right)\right\}\\
&=\PP\left\{\sum_{j,k\in J}\left(\frac{1}{N}\sum_{i=i_0-(N-1)}^{i_0}X_{ij}X_{ik}-\EE\left[X_{1j}X_{1k}\right]\right)^2> \frac{1}{4}\lambda^2_{\min}\left(\EE\left[X_{1J}^{}X_{1J}^T\right]\right)\right\}\leq \frac{1}{5}\;.
\end{align*}
So, $\PP\left\{E^{(i_0)}\right\}\geq 0.8$. 

Next we consider $F^{(i_0)}$. For all $j$, by Markov's inequality,
$$\PP\left\{\left|X_{1j}\right|^2>10q\sqrt{m}N\right\}\leq \frac{\EE\left[\left|X_{1j}\right|^2\right]}{10q\sqrt{m}N}\leq \frac{\sqrt{\EE\left[\left|X_{1j}\right|^4\right]}}{10q\sqrt{m}N}\leq (10qN)^{-1}\;$$
Then 
\begin{align*}
&\PP\left\{\left(F^{(i_0)}\right)^c\right\}= \PP\left\{\exists i\in\{i_0-(N-1),\dots,i_0\},j\in J, \text{ s.t. }X_{ij}^2>10q\sqrt{m}N\right\}\\
&\leq \sum_{i=i_0-(N-1)}^{i_0}\sum_{j\in J}\PP\left\{X_{ij}^2>10q\sqrt{m}N\right\}\leq 2qN\cdot (10qN)^{-1}=0.2\;.\end{align*}

Finally, for each $i_0=N,2N,3N,\dots,(2n')N$,
$$\PP\left\{E^{(i_0)}\cap F^{(i_0)}\right\}\geq 1-\PP\left\{\left(E^{(i_0)}\right)^c\right\}-\PP\left\{\left(F^{(i_0)}\right)^c\right\}\geq 0.6\;.$$
By the Chernoff bound, for sufficiently large $n$ (so that the relative difference between $\frac{n}{2N}$ and $n'=\left\lfloor\frac{n}{2N}\right\rfloor$ is sufficiently small),
\begin{align*}
\PP\left\{\#\{i_0:E^{(i_0)}\cap F^{(i_0)}\}<\frac{n}{2N}\right\}
&\leq \PP\left\{\mathrm{Binomial}\left(2n',0.6\right)<0.6\cdot\frac{n}{N}\cdot \left(1-\frac{1}{6}\right)\right\}\\
&\leq \PP\left\{\mathrm{Binomial}\left(2n',0.6\right)<0.6(2n')\cdot \left(1-\frac{1}{6.5}\right)\right\}\\
&\leq \exp\left\{-0.6(2n')\cdot \frac{\left(\frac{1}{6.5}\right)^2}{2}\right\}\\
&\leq  \exp\left\{-n\cdot (150N)^{-1}\right\}\;.
\end{align*}
So, for a fixed $J$ with $|J|=2q$,  with probability at least $1-e^{-\left(150\cdot \left\lceil80q^2ma_1^{-2}\right\rceil\right)^{-1}n}$,
$$\#\{i_0:E^{(i_0)}\cap F^{(i_0)}\}\geq 0.5\frac{n}{N}\;.\qedhere$$
\end{proof}

\end{document}